\documentclass[11pt,thmsa]{article}

\usepackage{amssymb,latexsym}
\usepackage{amsmath,amscd}
\usepackage[all]{xy}

 \usepackage[latin1]{inputenc}
 \usepackage{mathrsfs}
 \usepackage{dsfont}
 \usepackage{amsmath}
 \usepackage{amsthm}
 \usepackage{amsfonts}
 \usepackage{amssymb}
 \usepackage[dvips]{graphicx}
 \usepackage{wrapfig}
 \usepackage{layout}
 \usepackage{verbatim}
 \usepackage{alltt}
 \usepackage{xypic}
 \usepackage{bbm}
 \input xy
 \xyoption{all}

 \DeclareMathAlphabet{\mathpzc}{OT1}{pzc}{m}{it}

 \newtheorem{theorem}{Theorem}[section]
 \newtheorem{lemma}[theorem]{Lemma}
 \newtheorem{proposition}[theorem]{Proposition}
 
 \newtheorem{corollary}[theorem]{Corollary}
 \newtheorem{definition}[theorem]{Definition}
  \theoremstyle{definition}
 
 \newtheorem{remark}[theorem]{Remark}

\newtheorem*{acknowledgements}{Acknowledgements}
\renewenvironment{proof}{\noindent{\it
Proof.}}{\bgroup\hspace{\stretch{1}}$\square$\egroup\medskip\par}
\newcommand{\rmap}{\longrightarrow}

\newcommand{\RRep}{\mathcal{R}\textrm{ep}^{\infty}}
\newcommand{\URRep}{\mathcal{\hat{R}}\textrm{ep}^{\infty}}

\newcommand{\Ad}{\mathrm{Ad}}
\newcommand{\ad}{\mathrm{ad}}
\newcommand{\hPsi}{\hat{\Psi}}
\newcommand{\hR}{\hat{R}}

\newcommand{\id}{\mathrm{id}}
\newcommand{\End}{\textrm{End}}
\newcommand{\Hom}{\textrm{Hom}}

\begin{document}
\vspace{15cm}
 \title{Deformations of Lie brackets and representations up to homotopy}
\author{Camilo Arias Abad\footnote{Institut f\"ur Mathematik, Universit\"at Z\"urich,
camilo.arias.abad@math.uzh.ch. Partially supported
by SNF Grant 20-113439. } \hspace{0cm} and
Florian Sch\"atz\footnote{Center for Mathematical Analysis, Geometry and Dynamical Systems, IST Lisbon, fschaetz@math.ist.utl.pt. Partially supported by the research grant of the University of Z\"urich, the SNF-grant 20-121640, a FCT grant and project PTDC/MAT/098936/2008.}}


 \maketitle
 \begin{abstract} 
We show that representations up to homotopy can be differentiated in 
 a functorial way. A van Est type isomorphism theorem is established and used to prove
 a conjecture of Crainic and Moerdijk on deformations of Lie brackets.
\end{abstract}

\tableofcontents

\section{Introduction}\label{s:intro}

Some of the usual constructions of representations of Lie groups
and Lie algebras can be extended to the case of Lie groupoids and algebroids only in the context
of representations up to homotopy. These are $A_\infty$-, respectively $L_\infty$-, versions of usual representations. 
For Lie groupoids, the adjoint representation is a representation
up to homotopy which has the formal properties one would expect, for instance with respect to the cohomology of
the classifying space (\cite{AC2}). For Lie algebroids, the adjoint representation is a representation up to homotopy (\cite{AC1,Mehta}),
whose associated cohomology controls the infinitesimal deformations of the structure, as expected from the case of Lie algebras.

The purpose of this paper is to compare the global and the infinitesimal versions of representations up to homotopy, and to explain an application
regarding the deformations of Lie brackets.
We construct a differentiation functor $\Psi: \URRep(G) \rightarrow \RRep(A)$ from the category of unital representations up
to homotopy of a Lie groupoid to the category of representations up to homotopy of its Lie algebroid. Moreover, we show that given
a representation up to homotopy $E$ of $G$ there is a homomorphism
\[\Psi: H(G,E) \rightarrow H(A,\Psi(E))\]
from the cohomology associated to $E$ to the cohomology associated to $\Psi(E)$. We prove a van Est type theorem, which provides
conditions under which this map is an isomorphism (in certain degrees).

In \cite{CrainicMoerdijk} Crainic and Moerdijk introduced a deformation complex for Lie algebroids. Its cohomology controls
the deformations
of the Lie algebroid structures and arises in the study of the stability of leaves of Lie algebroids (\cite{CrainicFernandes}). Based 
on the rigidity properties of compact group actions, Crainic and Moerdijk stated a rigidity conjecture (Conjecture $1$ of \cite{CrainicMoerdijk}) which gives
conditions under which the deformation cohomology should vanish. In the case of a Lie algebra $\mathfrak{g}$, 
the conjecture corresponds to the vanishing of $H^2(\mathfrak{g},\mathfrak{g})$, for $\mathfrak{g}$ semisimple of compact type.
We will explain how representations up to homotopy can be used to give a proof of this conjecture.
Since the deformation cohomology of a Lie algebroid coincides with the cohomology associated to the adjoint representation, 
one can reproduce
the usual proof for Lie algebras (\cite{VanEst1,VanEst2,VanEst3}), once the differentiation functor and the van Est isomorphism theorem have been established
in the context of representations up to homotopy.

The paper is organized as follows.  Section \S \ref{preliminaries} contains the definitions and general facts
about representations up to homotopy.  In Section \S \ref{s:VanEst} we construct the differentiation functor for representations up to homotopy.
The main result of this section is Theorem \ref{theorem:differentiation_reps_up_to_homotopy}. 
We begin Section \S \ref{section:deformations} by proving a van Est isomorphism theorem, see Theorem \ref{isomorphism}.
We then show that by differentiating the adjoint representation of a Lie groupoid one obtains that of the Lie algebroid (Proposition \ref{Ad=ad}).
At the end we prove the rigidity conjecture of \cite{CrainicMoerdijk}, which is Theorem \ref{conjecture}.

\begin{acknowledgements}We would like to thank Alberto Cattaneo, Marius Crainic and Ionut Marcut for invitations to Z\"urich and Utrecht, where
parts of this work were done. Camilo Arias Abad also thanks Marius Crainic for many discussions on this subject.
\end{acknowledgements}

\section{Preliminaries}\label{preliminaries}

In order to fix our conventions, we will review the definitions and basic facts regarding
representations up to homotopy of Lie algebroids and groupoids. More details on these constructions as well as the proofs of the
results stated in this section can be found in
\cite{AC1,AC2}. Throughout the text, $A$ will denote a Lie algebroid over a manifold
$M$, and $G$ will be a Lie groupoid over $M$.

Given a Lie algebroid $A$, there is a differential graded algebra
$\Omega(A)= \Gamma(\Lambda A^*)$, with differential defined via
the Koszul formula
\begin{eqnarray*}
d \omega(\alpha_1, \dots ,\alpha_{n+1})&=& \sum_{i<j}(-1)^{i+j} \omega([\alpha_i,\alpha_j],\cdots,\hat{\alpha}_i,\dots,\hat{\alpha}_j,\dots,\alpha_{k+1})\\
&&+\sum_i(-1)^{i+1}L_{\rho(\alpha_i)}\omega(\alpha_1,\dots,\hat{\alpha}_i,\dots,\alpha_{k+1}),
\end{eqnarray*}
where $\rho$ denotes the anchor map and $L_{X}(f)= X(f)$ is the Lie derivative along vector fields.
The operator $d$ is a coboundary operator ($d^{2}= 0$) and satisfies
the derivation rule
\[ d(\omega\eta)= d(\omega)\eta+ (-1)^p\omega d(\eta),\]
for all $\omega\in \Omega^p(A), \eta\in \Omega^q(A)$.

\begin{definition}
Let $A$ be a Lie algebroid over $M$. An $A$-connection on a vector
bundle $E$ over $M$ is a bilinear map $\nabla:\Gamma(A)\times
\Gamma(E) \rightarrow \Gamma(E)$, $(\alpha,S) \mapsto
\nabla_{\alpha}(S)$ such that
\[  \nabla_{f\alpha}(s) = f\nabla_{\alpha}(s)\quad \textrm{and} \quad \nabla _{\alpha}(fs)=f\nabla_{\alpha}(s)+(L_{\rho(\alpha)}f)s\]
hold for all $f\in C^{\infty}(M)$, $s\in \Gamma(E)$ and $ \alpha \in
\Gamma(A)$. The curvature of $\nabla$ is the tensor given by
\begin{equation*}
R_{\nabla}(\alpha, \beta)(s):= \nabla_{\alpha}\nabla_{\beta}(s)-
\nabla_{\beta}\nabla_{\alpha}(s)- \nabla_{[\alpha, \beta]}(s)
\end{equation*}
 for $\alpha,\beta \in \Gamma(A)$, $s \in \Gamma(E)$.

The $A$-connection $\nabla$ is called flat if $R_{\nabla}= 0$. A
representation of $A$ is a vector bundle $E$ together with a flat
$A$-connection $\nabla$ on $E$.
\end{definition}
Given a graded vector bundle $E=\bigoplus_{k\in \mathbb{Z}}E^k $ over $M$, we denote by $\Omega(A,E)$ the space $\Gamma(\Lambda(A^*) \otimes E)$,
graded with respect to the total degree. The wedge product gives this space the structure of a graded commutative module over the algebra $\Omega(A)$. 
This means that $\Omega(A,E)$ is a bimodule over $\Omega(A)$ and that
\begin{equation*}
\omega \wedge \eta=(-1)^{kp} \eta \wedge \omega
\end{equation*}
holds for $\omega \in \Omega^k(A)$ and $\eta \in \Omega(A,E)^p$. 
When doing the wedge products, it is important that one takes the grading into account. 
See \cite{AC1} a detailed explanation of the sign conventions used in the defintion in the wedge product. 
In order to simplify the notation we will sometimes omit the wedge symbol.
\begin{definition} A representation up to homotopy of $A$ consists of a graded vector bundle $E$ over $M$
and a linear operator
\begin{equation*}
D:\Omega(A,E)\rightarrow \Omega(A,E)
\end{equation*}
which increases the total degree by one and satisfies $D^2= 0$ as well as
the graded derivation rule
\begin{equation*}
D(\omega \eta)=d(\omega) \eta +(-1)^k \omega D(\eta)
\end{equation*}
for all $\omega\in \Omega^k(A)$, $\eta\in \Omega(A, E)$. The cohomology of the resulting complex is denoted by
$H(A, E)$.
\end{definition}
Intuitively, a representation up to homotopy of $A$ is a cochain complex of
vector bundles over $M$ endowed with an $A$-connection which is flat up to homotopy. This is made precise as follows:

\begin{proposition}\label{Rep-long}
There is a bijective correspondence between representations up to homotopy
$(E, D)$ of $A$ and graded vector bundles $E$ over $M$ equipped with
\begin{enumerate}
\item A degree $1$ operator $\partial$ on $E$ making $(E, \partial)$ a cochain complex of vector bundles.
\item An $A$-connection $\nabla$ on $E$ which respects the degree and commutes with $\partial$.

\item For each $i> 1$ an $\mathrm{End(E)}$-valued $i$-cochain $\omega_{i}$ of total degree 1, i.e.
\[ \omega_{i}\in \Omega^i(A, \mathrm{End}^{1-i}(E))\]
satisfying

\[ \partial(\omega_{i})+ d_{\nabla}(\omega_{i-1})+ \omega_{2}\wedge \omega_{i-2}+
\omega_{3}\wedge \omega_{i-3}+ \ldots + \omega_{i-2}\wedge \omega_{2}=
0.\]
\end{enumerate}
The correspondence is characterized by
\[ D(\eta)= \partial(\eta)+ d_{\nabla}(\eta)+ \omega_{2}\wedge \eta+ \omega_{3}\wedge \eta+ \ldots .\]
\end{proposition}

We also write
\begin{equation}
\label{structure} D= \partial+ \nabla + \omega_{2}+ \omega_{3}+
\ldots .
\end{equation}

There is a natural notion of morphism between representations
up to homotopy.

\begin{definition} A morphism $\phi: E\to F$ between two representations up to homotopy
of $A$ is a linear degree zero map
\begin{equation*}
\phi:\Omega(A,E) \rightarrow \Omega(A,F)
\end{equation*}
which is $\Omega(A)$-linear and commutes with the structure
differentials $D_E$ and $D_F$. 

We denote by $\RRep(A)$ the resulting category.
\end{definition}

Any morphism is necessarily of the form
\begin{equation*}
\phi=\phi_0 +\phi_1+\phi_2 + \cdots
\end{equation*}
where $\phi_i$ is a $\textrm{Hom}(E, F)$-valued $i$-form on $A$ of
total degree zero, i.e.
\[ \phi_i \in \Omega^i(A, \textrm{Hom}^{-i}(E,F))\]
satisfying
\[ \partial(\phi_n)+ d_{\nabla}(\phi_{n-1})+ \sum_{i+j= n, i\geq 2} [\omega_{i}, \phi_{j}]= 0.\]

\begin{remark}
Let $E$ and $E'$ be representations up to homotopy of $A$. A quasi-isomorphism from $E$ to $E'$ is a morphism:
\begin{equation*}
\phi=\phi_0+\phi_1+\phi_2 +\cdots,
\end{equation*}
such that the morphism of complexes $\phi_0$ induces an isomorphism in cohomology at every point. 
Any quasi-isomorphism $\phi$ induces isomorphisms in cohomology:
\begin{equation*}
\phi:H(A,E) \rightarrow H(A,E).
\end{equation*}
\end{remark}

There is a natural representation up to homotopy associated to any Lie algebroid, called
the adjoint representation. It is determined by the choice of a connection $\nabla$
on the vector bundle $A$. Moreover, different choices of connection give naturally isomorphic representations up to homotopy.
\begin{definition}
Let $\nabla$ be a connection on a Lie algebroid $A$. The adjoint representation of $A$ induced
by $\nabla$ is the representation up to homotopy of $A$ on the vector bundle $A \oplus TM$, where $A$
is in degree zero and $TM$ in degree one, given by the following structure operator:
 \begin{equation*}
D=\partial+\nabla^{bas}+K_{\nabla}.
\end{equation*}
 The differential $\partial$ on the graded vector bundle is given by the anchor, $\nabla$ is an $A$-connection and $K_{\nabla}$ is an
 endomorphism valued cochain. The latter two are defined by the formulas
 \begin{eqnarray*}
 \nabla^{bas}_{\alpha}(\beta)&=& \nabla_{\rho(\beta)}(\alpha)+ [\alpha, \beta],\\
 \nabla^{bas}_{\alpha}(X)&=& \rho(\nabla_{X}(\alpha))+ [\rho(\alpha), X],\\
 K_{\nabla}(\alpha,\beta)(X)&=&\nabla_X([\alpha,\beta])-[\nabla_X(\alpha),\beta]-
[\alpha,\nabla_X(\beta)]-\nabla_{\nabla_{\beta}^{\textrm{bas}}X}(\alpha)
+\nabla_{\nabla_{\alpha}^{\textrm{bas}}X}(\beta).
\end{eqnarray*}
This representation up to homotopy is denoted by $\ad_{\nabla}(A)$. The isomorphism class of this
representation, which is independent of the connection, will be denoted by $\ad(A)$.
\end{definition}

Next we will discuss representations up to homotopy of groupoids.
Given a Lie groupoid $G$ over $M$,  we will write $s$ and $t$ for the source and target map
respectively.  We denote the space of sequences
$(g_1,\dots,g_k)$ of composable arrows in $G$ by $G_k$. Here we write the sequence in the order
of the composition of functions, namely $t(g_i)=s(g_{i-1})$. Since the source and target map
are required to be submersions, the spaces $G_k$ are smooth manifolds. We will also write $s$ and $t$ for
the maps 
\begin{align*}
s:G_k \rightarrow M=:G_0, \,(g_1.\dots,g_k)\mapsto s(g_k),\\
t:G_k \rightarrow M=:G_0, \, (g_1.\dots,g_k)\mapsto t(g_1).
\end{align*}

The differentiable cohomology of $G$, denoted $H_{\textrm{diff}}(G)$, or just $H(G)$,
is the cohomology of the complex
\[ \delta: C^k(G)\rmap C^{k+1}(G),\]
where $C^k(G)$ is the space of smooth valued functions on $G_k$, and the coboundary operator is
\begin{equation*}
\delta(f)(g_1,\dots ,g_{k+1})=(-1)^kf(g_2,\dots ,g_{k+1})+\sum_{i=1}^{k}(-1)^{i+k}f(g_1,\dots ,g_ig_{i+1},\dots,g_{k+1})-f(g_1,\dots,g_k).
\end{equation*}

The space
$C^{\bullet}(G )=\oplus_k C^k(G)$ has an algebra structure given by
\[ (f \star h)(g_1, \ldots , g_{k+p})=(-1)^{kp} f(g_1, \ldots , g_k) h(g_{k+1}, \ldots , g_{k+p}),\]
for $f\in C^k(G)$, $h\in C^p(G)$. In this way, $C^{\bullet}(G)$
becomes a DGA (differential graded algebra), i.e. the coboundary operator
$\delta$ is an algebra derivation
\[ \delta (f \star h)= \delta(f) \star h+ (-1)^k f \star \delta(h).\]
In particular, $H(G)$ inherits a graded algebra structure.

The space $\hat{C}(G)$ of normalized cochains is defined by
\begin{equation*}
\hat{C}^k(G)=\{f \in C^k(G): f(g_1,\dots, 1,\dots , g_k)=0 \textrm{ if } k>0\}.
\end{equation*}

The differential graded algebra structure on $C^{\bullet}(G)$ restricts to this subspace, 
and the inclusion $\hat{C}(G) \hookrightarrow C(G)$ induces
an isomorphism in cohomology.
Given a vector bundle $E$ over $M$, the vector space of $E$-valued
cochains $C^{\bullet}(G, E)$ is the graded vector space 
\begin{equation*}
C^{\bullet}(G, E)=\bigoplus_{k\in \mathbb{Z}} C^k(G,E),
\end{equation*}
where
\[ C^{k}(G, E):=\Gamma(G_k, t^*(E)).\]
As before, the subspace of normalized cochains is 
\begin{equation*}
\hat{C}^k(G,E)=\{\eta \in C^k(G,E): \eta (g_1,\dots, 1,\dots , g_k)=0 \textrm{ if } k>0\}.
\end{equation*}
For a graded vector bundle $E =\bigoplus_k E^k$, the space of $E$ valued cochains is denoted
by $ C(G, E)^{\bullet}$ and is graded with respect to the total degree. Explicitly:
\begin{align*}
C(G, E)^{k}:= \bigoplus_{i+j=k}C^i(G, E^j),\\
\hat{C}(G, E)^{k}:= \bigoplus_{i+j=k}\hat{C}^i(G, E^j).
\end{align*}
The space $C(G, E)^{\bullet}$ is a graded right module over the algebra
$C^{\bullet}(G)$: Given $f \in C^k(G)$ and $\eta\in C^p(G, E^q)$, their product $\eta \star f \in C^{k+p}(G, E^q)$ is defined by
\[ (\eta \star  f)(g_1, \ldots , g_{k+p})= (-1)^{kp}\eta(g_1, \ldots , g_p) f(g_{p+1}, \ldots , g_{p+k}).\]
The same formulas give $\hat{C}(G, E)^{\bullet}$ the structure of a right $\hat{C}^{\bullet}(G)$-module.

\begin{definition}\label{def-rep-up}
A representation up to homotopy of $G$ on a graded vector bundle $E$
over $M$ is a linear operator 
\[ D:C(G,E)^{\bullet}
\rightarrow C(G,E)^{\bullet+1}, \]
which increases the total degree by one and satisfies $D^2= 0$ as well as the 
Leibniz identity
\[ D(\eta \star f )= D(\eta)\star f+(-1)^k \eta \star  \delta(f)\]
for $\eta \in C(G,E)^k$ and  $f \in C^{\bullet}(G)$. We will denote the resulting cohomology by $H(G,E)$.
A morphism $\phi: E\rmap E'$ between two 
representations up to homotopy $E$ and $E'$ is a degree zero
$C^{\bullet}(G)$-linear map
\begin{equation*}
\phi:C(G,E)^{\bullet}\rightarrow C(G,E')^{\bullet},
\end{equation*}
that commutes with the structure operators of $E$ and $E'$. 
We denote the resulting category by 
$\RRep(G)$. 

We call a representation up to homotopy unital if the 
operator $D$ preserves the normalized subcomplex. A morphism of unital representations up to homotopy is a morphism of
representations up to homotopy that preserves the normalized subcomplexes. We denote by $\URRep$ the resulting category of
unital representations up to homotopy.
\end{definition}

In order to describe the structure of these representations up to homotopy, we consider the bigraded vector space 
$C_G(\End(E))$ which, in 
bidegree $(k,l)$, is
\begin{equation*}
C^k_G(\End^l(E))=\Gamma(G_k,\Hom(s^*(E^{\bullet}),t^*(E^{\bullet+l}))).
\end{equation*}
Recall that $s$ and $t$ are the maps $s(g_1,\dots,g_k)=s(g_k)$ and
$t(g_1,\dots,g_k)=t(g_1)$ respectively. 
This space has the structure of a bigraded algebra as follows: The product of $F \in C^k_G(\End^l(E))$ and 
$F' \in C^{k'}_G(\End^{l'}(E))$ is the element $F \circ F'$ of $C^{k+k'}_G(\End^{l+l'}(E))$ defined by the formula
\begin{equation*}
(F \circ F')(g_1,\dots,g_{k+k'})=F(g_1,\dots,g_k)\circ F'(g_{k+1},\dots,g_{k+k'}).
\end{equation*}

Similarly, we consider the bigraded space 
\begin{equation*}
C_G( \Hom(E, E'))=\Gamma(G_k,\Hom(s^*(E^{\bullet}),t^*(E'^{\bullet+l}))),
\end{equation*}
for any two graded vector bundles $E$ and $E'$. By the same formula as above, this space is a $C_G(\End(E'))$-$C_G(\End(E))$ bimodule.

The normalized versions of these spaces are defined as follows:
\begin{eqnarray*}
& \hat{C}^k_G(\End^l(E))&=\{F \in C^k_G(\End^l(E)): \eta(g_1,\dots ,1,\dots, g_k)=0 \textrm{ if } k>0 \},\\
&\hat{C}^k_G(\Hom^l(E,E'))&=\{F \in C^k_G(\Hom^l(E,E')): \eta(g_1,\dots ,1,\dots, g_k)=0 \textrm{ if } k>0 \}.
\end{eqnarray*}

We will also consider the space $\overline{C}^1_G(\End(E))$ defined by
\begin{equation*}
\overline{C}^1_G(\End^0(E))=\{\eta \in C^1_G(\End^0(E)): \eta(1)=\id \}.
\end{equation*}

\begin{remark}  Given $F\in C^k_G( \End^m(E))$  we define operators
\[\tilde{F}: C(G,E)\rightarrow C(G,E),\]
by the following formulas: For $k=1,m=0$ and $\eta \in C^p(G,E^l)$ we set
\begin{eqnarray}
\tilde{F}(\eta)(g_1, \dots , g_{p+1}) & = & (-1)^{p+l}\Big(F(g_1)\eta(g_2, \ldots , g_{p+1})+ \nonumber \\
                                                               &  + & \sum_{i= 1}^{p}(-1)^i F(g_1, \ldots, g_ig_{i+1}, \ldots , g_{p+1})+ (-1)^{p+1} F(g_1, \ldots, g_p)\Big).  \nonumber
\end{eqnarray}
If $k\neq 1$ or $m\neq 0$, the operator is defined by
\[\tilde{F}(\eta)(g_1,\dots,g_{p+k})=(-1)^{k(p+l)}F(g_1,\dots,g_k)(\eta(g_{k+1},\dots,g_{k+p})).\]
For $k=1$ and $m=0$, the operator $\tilde{F}$ is a graded derivation with respect to the right $C(G)$-module structure and respects the normalized subcomplex if and
only if $F \in \overline{C}^1_G( \End^0(E))$. For $k\neq1$ or $m\neq 0$, the operator $\tilde{F}$ is a map of $C(G)$-modules and it respects the normalized subcomplex if and
only if $F \in \hat{C}^k_G( \End^m(E))$. 
\end{remark}

\begin{proposition}\label{decomposition with respect to F}
There is a bijective correspondence between representations up to
homotopy of $G$ on the graded vector bundle $E$ and sequences
$(F_k)_{k\geq 0}$ of elements $F_k \in C^{k}_G( \mathrm{End}^{1-k}(E))$
which, for all $k\geq 0$, satisfy
\begin{equation}\label{structure equations}
\sum_{j=1}^{k-1} (-1)^{j} F_{k-1}(g_1, \ldots , g_jg_{j+1}, \ldots, g_k)  = 
\sum_{j=0}^{k} (-1)^{j} F_{j}(g_1, \ldots , g_j) \circ F_{k-j}(g_{j+1}, \ldots , g_k). 
\end{equation}
The correspondence is characterized by
\[ D=\tilde{F_0}+\tilde{F_1}+\tilde{F_2}+\dots.\]

A representation is unital if and only if $F_1 \in \overline{C}^1_G(\End^0(E))$ and
$F_k\in \hat{C}^{k}_G( \mathrm{End}^{1-k}(E))$ for $k \neq 1$.

Let $E$ and $E'$ be two representations up to homotopy  with corresponding sequences of structure maps
$(F_k)$ and $(F_{k}^{'})$. There is a bijective correspondence between 
morphisms  of representations up to homotopy $\phi: E\rmap E'$ and sequences
$(\phi_k)_{k\geq 0}$ of elements $\phi_k\in C^k_G(\Hom^{-k}(E,E'))$
which, for all $k\geq 0$, satisfy
\begin{eqnarray}\label{equations for a map}
\sum_{i+j=k}(-1)^{j}\phi_j(g_1,\dots,g_j) \circ
F_{i}(g_{j+1},\dots,g_k)&=& \sum_{i+j=k}F_{j}^{'}(g_1,\dots , g_j)
\circ \phi_{i}(g_{j+1},\dots,g_k)\\
&&+\sum_{j=1}^{k-1}(-1)^{j}\phi_{k-1}(g_1,\dots,g_jg_{j+1},\dots
,g_k).\nonumber
\end{eqnarray}
If $E$ and $E'$ are unital representations up to homotopy, then a morphism between them
is a unital morphism if and only if $\phi_k\in \hat{C}^k_G(\Hom^{-k}(E,E'))$ for $k\geq 0$.
\end{proposition}

\begin{remark}
Let $E$ and $E'$ be representations up to homotopy of $G$. A quasi-isomorphism from $E$ to $E'$ is a morphism:
\begin{equation*}
\phi=\phi_0+\phi_1+\phi_2 +\cdots,
\end{equation*}
such that the morphism of complexes $\phi_0$ induces an isomorphism in cohomology at every point. 
Any quasi-isomorphism $\phi$ induces isomorphisms in cohomology:
\begin{equation*}
\phi:H(G,E) \rightarrow H(G,E').
\end{equation*}
\end{remark}

\begin{remark}
We have defined unital representations up to homotopy as differentials on the space of cochains which restrict to the
normalized cochains. On the other hand, the proof of Proposition \ref{decomposition with respect to F} shows that a differential on the
normalized cochains extends naturally to a differential on all cochains. Thus, one can take the equivalent point
of view that a unital representation is a degree one derivation on the space of normalized cochains which squares to zero.
Note that this does not conflict with the definition of cohomology: for any unital representation up to homotopy, the inclusion of
normalized cochains in the space of cochains induces an isomorphism in cohomology.  In what follows we will regard unital representations up to homotopy as differentials on the space of normalized cochains. From now on all representations up to homotopy are assumed to be unital, even when we do not mention it explicitly.
\end{remark}
In order to define the adjoint representation of a Lie groupoid, one needs to choose a connection. 
As before, two different connections give naturally isomorphic representations. Let us first explain what
we mean by a connection in this context.

\begin{definition} An Ehresmann connection on a Lie groupoid $G$ is a sub-bundle $\mathcal{H}\subset TG$ which is complementary
to the kernel of $ds$ and has the property that
\[ \mathcal{H}_{x}= T_xM, \ \ \forall\ x\in M \subset G.\]
\end{definition}

There is an equivalent way of looking at connections which uses the vector bundle 
underlying the Lie algebroid of $G$, i.e.
\[ A= \textrm{Ker}(ds)|_{M}.\]
The construction of the Lie algebroid shows that there is  a short exact sequence
of vector bundles over $M$:
\[ t^*A\stackrel{r}{\rmap} TG\stackrel{ds}{\rmap} s^*TM,\]
where $r$ is given by right translations as follows: For $g: x\rmap y$ in $G$, $r_{g}$ is the derivative 
of the right multiplication by $g$, i.e.
\begin{equation*}
r_g= (dR_g)_{1_y}: A_y\rmap T_g G.
\end{equation*} 
With this, we see that the following structures are equivalent:
\begin{itemize}
\item An Ehresmann connection $\mathcal{H}$ on $G$.
\item  A right splitting of the previous sequence, i.e. a section $\sigma: s^*TM\rmap TG$ 
of $(ds)$, which restricts to the natural splitting at the identities.
\item A left splitting $\omega$ of the previous sequence which restricts to the natural one at the identities. Such a splitting can be viewed as a 1-form $\omega\in \Omega^1(G, t^*A)$ 
satisfying $\omega(r(\alpha))= \alpha$ for all $\alpha$.
\end{itemize}
These objects are related by
\[ \omega\circ r= \textrm{Id}, \ (ds)\circ \sigma= \textrm{Id}, \ r\circ \omega+ \sigma\circ (ds)= \textrm{Id}, \ \mathcal{H}= \textrm{Ker}(\sigma)= \textrm{Im}(\omega).\]

From now on, when talking about an Ehresmann connection on $G$, we will make no distinction between the sub-bundle $\mathcal{H}$, 
the splitting $\sigma$ and the form $\omega$. In particular, we will also say that $\sigma$ is a connection on $G$.
Moreover, we will often say connection instead of Ehresmann connection. Observe that an easy partition of unity argument shows that any Lie groupoid
admits an Ehresmann connection.

There is another version of the definition of connections which uses the target map instead of the source.
The two are equivalent: any $\mathcal{H}$ as above induces a sub-bundle
\[ \overline{\mathcal{H}}:= (d\iota)(\mathcal{H})\]
which is complementary to $\textrm{Ker}(dt)$ and, at points $x\in M$, coincides with the image of $T_{x}M$.
Similarly,  there is  a short exact sequence
of vector bundles over $M$:
\[ s^*A\stackrel{l}{\rmap} TG\stackrel{dt}{\rmap} t^*TM\]
where $l$ is given by left translations, and a connection $\mathcal{H}$ is the same as a splitting
$(\overline{\omega}, \overline{\sigma})$ of this sequence. In terms of $(\omega, \sigma)$, one has
\[ \overline{\sigma}_{g}(v)= (d\iota)_{g^{-1}}(\sigma_{g^{-1}}(v)),\ \overline{\omega}_g(X)= \omega_{g^{-1}}(d\iota)_{g}(X).\]

\begin{definition}
Let $\sigma$ be a connection on a Lie groupoid $G$. The adjoint representation of $A$ induced
by $\sigma$ is the representation up to homotopy of $G$ on the graded vector bundle $\Ad=A \oplus TM$, where $A$
is in degree zero and $TM$ in degree one, given by the following structure operator:
 \begin{equation*}
D=\partial+\lambda+K_{\sigma},
\end{equation*}
where $\partial$ is the anchor while $\lambda \in \overline{C}^1_G(\End^0(\Ad))$ and $K_\sigma \in  \hat{C}^2_G(\End^{-1}(\Ad))$ are
\begin{eqnarray*}
 \lambda_g(X)&=& (dt)_g(\sigma_g(X)),\\
 \lambda_g(\alpha)&=& - \omega_g(l_g(\alpha)),\\
K_{\sigma}(g,h)(X)&=&(dm)_{g, h}(\sigma_g(\lambda_h(X)), \sigma_h(X))-
\sigma_{gh}(X).
 \end{eqnarray*}
We denote this representation up to homotopy by $\Ad_{\sigma}(G)$. The isomorphism class of this
representation, which is independent of the connection, will be denoted by $\Ad(G).$
\end{definition}

\section{Differentiation of representations up to homotopy}\label{s:VanEst}

Here we will construct a differentiation functor
\begin{eqnarray*}
 \Psi:\URRep(G) \rightarrow \RRep(A)
\end{eqnarray*}
from the category of unital representations up to homotopy of a Lie groupoid to those
of the Lie algebroid. We prove that for a representation up to homotopy $E$ of 
a Lie groupoid $G$, differentiation yields a van Est homomorphism $\Psi:\hat{C}(G,E) \rightarrow \Omega(A,\Psi(E))$.

\subsection{Differentiating cochains}\label{ss:differentiating_cochains}

We will first explain how to differentiate Lie groupoid cochains to obtain Lie algebroid cochains.
The computations already appeared in \cite{Crainic}. We reproduce them here in order to adapt the sign
conventions to the case of graded vector bundles. In what follows, $G$ is a Lie groupoid over the manifold $M$ with
Lie algebroid $A$.
\begin{definition}
Given a section $\alpha \in \Gamma(A)$, there is a map $R_{\alpha}: \hat{C}^{k+1}(G,E) \to \hat{C}^{k}(G,E)$ defined by the formula
\begin{eqnarray*}
R_{\alpha}(\eta)(g_{1},\dots,g_{k}):=\left. \frac{d}{d\epsilon}\right|_{\epsilon=0}\eta(g_1,\dots,g_{k},\Phi^{\alpha}_{\epsilon}(s(g_{k}))^{-1}).
\end{eqnarray*}
Here $\Phi^{\alpha}_{\epsilon}$ denotes the flow on $G$ of the right invariant vector field on $G$ associated to $\alpha$.
\end{definition}
The map $R_{\alpha}$ can alternatively be described as follows: consider $\eta(g_1,\dots,g_{k},(-)^{-1})$ as a smooth map
from $s^{-1}(s(g_{k}))$ to $E_{t(g_{1})}$.  Differentiating this map at the point $s(g_k)$, we obtain a linear map
from $A_{s(g_{k})}$ to $E_{t(g_{1})}$. If one applies this map to $\alpha(s(g_{k}))$ the result is $R_{\alpha}(\eta)(g_{1},\dots,g_{k})$.
These two descriptions of $R_{\alpha}$ immediately imply:

\begin{lemma}\label{lemma:R}
For $\alpha \in \Gamma(A)$, $h\in C^{\infty}(M)$, $\eta \in \hat{C}^{k}(G,E)$ and $f \in \hat{C}^{l}(G)$
the following identities hold:
\begin{itemize}
\item[a)] $R_{h\alpha}(\eta)=R_{\alpha}(\eta) \star h$.
\item[b)] If $l>0$: $R_{\alpha}(\eta \star f) = (-1)^{k}\eta \star R_{\alpha}(f)$.
\item[c)] If $l=0$: $R_{\alpha}(\eta \star f) = R_{\alpha}(\eta) \star f$.
\end{itemize}
\end{lemma}

We observe that the last identity is true because we consider normalized cochains. 
\begin{definition}\label{def:differentiation_cochains}
We define the differentiation map $\Psi: \hat{C}^{k}(G,E^l) \to \Omega^{k}(A,E^l)$ by the formula
\begin{eqnarray*}
\Psi(\eta)(\alpha_1,\dots,\alpha_k):=(-1)^{kl}\sum_{\sigma \in S_k}(-1)^{|\sigma|}R_{\alpha_{\sigma(k)}}\cdots R_{\alpha_{\sigma(1)}}\eta
\end{eqnarray*}
if $k>0$ and to be equal to the identity if $k=0$.
\end{definition}

\begin{lemma}\label{lemma:differentiation_cochains}

 The map $\Psi$ is well-defined, namely $\Psi(\eta)$ is multilinear with respect to functions on $M$.
In the case of trivial coefficients, the map $\Psi:\hat{C}(G)\rightarrow \Omega(A)$ is an algebra map. In general,
$\Psi$ respects the module structure in the sense that for  $\eta \in \hat{C}^{k}(G,E^l)$ and $f \in \hat{C}^{m}(G)$
\begin{eqnarray*}
\Psi(\eta\star f)= \Psi(\eta)\Psi(f)
\end{eqnarray*}
holds.

\end{lemma}

\begin{proof}
The fact that the map
\begin{eqnarray*}
\Gamma(A)\times \dots \times \Gamma(A) & \to & \Gamma(A)\\
(\alpha_1,\dots,\alpha_k) & \mapsto & \sum_{\sigma \in \Sigma_k}(-1)^{|\sigma|}R_{\alpha_{\sigma(k)}}\cdots R_{\alpha_{\sigma(1)}}\eta
\end{eqnarray*}
is $C^{\infty}(M)$-multilinear, follows from parts $a)$ and $c)$ of  Lemma \ref{lemma:R}.
The second assertion is a particular case of the third one, which follows from an easy computation.
\end{proof}

The sequence $(G_k)_{k\geq 0}$ is a simplicial manifold called the nerve of $G$. The face maps of the simplicial structure are 
given by the formulas
\begin{eqnarray}\label{simplicial faces}
d_i(g_1,\dots,g_k):=
\begin{cases}
(g_2,\dots,g_k) & \textrm{for } i=0,\\
(g_{1},\dots,g_{i}g_{i+1},\dots,g_k) & \textrm{for } 0<i<k,\\
(g_1,\dots,g_{k-1}) & \textrm{for } i=k.
\end{cases}
\end{eqnarray}

\begin{lemma}\label{lemma:R2}
For $f \in \hat{C}^{k}(G)$ the following identities hold:
\begin{enumerate}
\item[a)] $R_{\alpha}(d_{0}^{*}f)=d_{0}^{*}(R_{\alpha}f)$ if $k>0$, 
\item[b)] $R_{\alpha}(d_{0}^{*}f)=L_{\rho(\alpha)}f$ if $k=0$,
\item[c)] $R_{\alpha_{k+1}}\cdots R_{\alpha_1}(d_{0}^{*}f)=L_{\rho(\alpha_{k+1})}(R_{\alpha_{k}}\cdots R_{\alpha_1}f)$,
\item[d)] $R_{\alpha}(d_{k+1}^{*}f)=0$,
\item[e)] $R_{\alpha}(d_{i}^{*}f)=d_{i}^{*}(R_{\alpha}f)$ for $0<i<k$,
\item[f)] $R_{[\alpha,\beta]}f = R_{\alpha}R_{\beta}(d_{k}^{*}f) - R_{\beta}R_{\alpha}(d_{k}^{*}f)$.
\end{enumerate}
\end{lemma}

\begin{proof}
All but the last assertion follow from easy calculations. By fixing the first $k-1$ arguments, the last claim can be reduced to the case $\eta \in \hat{C}^{1}(G)$. For this case it is enough to show that
\begin{eqnarray*}
\left. \frac{d}{d\epsilon}\right|_{\epsilon =0}\eta(\Phi^{[\alpha,\beta]}_{\epsilon}(x))=\left. \frac{d}{d\tau}\frac{d}{d\epsilon}\right|_{\tau=0,\epsilon=0}\left( \eta\left(\Phi^{\beta}_{\epsilon}(t(\Phi^{\alpha}_{\tau}(x)))\bullet \Phi^{\alpha}_{\tau}(x)\right) 
-\eta\left(\Phi^{\alpha}_{\epsilon}(t(\Phi^{\beta}_{\tau}(x)))\bullet \Phi^{\beta}_{\tau}(x)\right)  \right).
\end{eqnarray*}
Here $\bullet$ denotes the composition of the Lie groupoid $G$. 
Since $\Phi_\epsilon^\alpha$  and $\Phi_\epsilon^\beta$  are flows of right invariant vector fields, we know that
\begin{eqnarray*}
\Phi^{\beta}_{\epsilon}(t(\Phi^{\alpha}_{\tau}(x)))\bullet \Phi^{\alpha}_{\tau}(x)=\Phi^{\beta}_{\epsilon}(\Phi^{\alpha}_{\tau}(x)),\\
\Phi^{\alpha}_{\epsilon}(t(\Phi^{\beta}_{\tau}(x)))\bullet \Phi^{\beta}_{\tau}(x)=\Phi^{\alpha}_{\epsilon}(\Phi^{\beta}_{\tau}(x)).
\end{eqnarray*}
Inserting these two equations in the first equation above, we obtain
\begin{eqnarray*}
L_{\widehat{[\alpha, \beta]}}\eta(x)=L_{\hat{\alpha}}L_{\hat{\beta}}\eta (x) - L_{\hat{\beta}}L_{\hat{\alpha}}\eta(x),
\end{eqnarray*}
where the $\hat{\alpha}$ and $\hat{\beta}$  denote the right invariant vector fields associated to $\alpha$ and $\beta$.
\end{proof}

\begin{proposition}\label{proposition:Phi}
The map $\Psi: (\hat{C}^{\bullet}(G),\delta) \to (\Omega^{\bullet}(A),d)$ is a morphism of differential graded algebras.
\end{proposition}

\begin{proof}
It remains to check that $\Psi$ is a chain map. To this end we pick $f \in \hat{C}^{k}(G)$ and compute
\begin{eqnarray*}
\Psi(\delta f)(\alpha_1,\dots,\alpha_{k+1})&=&\sum_{\sigma \in S_{k+1}}(-1)^{|\sigma|} R_{\alpha_{\sigma(k+1)}}\cdots R_{\alpha_{\sigma(1)}}(\delta(f))\\
&=&\sum_{i=0}^{k+1}\sum_{\sigma \in S_{k+1}}(-1)^{i+k+|\sigma|} R_{\alpha_{\sigma(k+1)}}\cdots R_{\alpha_{\sigma(1)}}(d_{i}^{*}f)\\
&=& \sum_{\sigma\in S_{k+1}}(-1)^{|\sigma|+k}R_{\alpha_{\sigma(k+1)}}\cdots R_{\alpha_{\sigma(1)}}(d_0^{*}f)\\
&& + \sum_{i=1}^{k}\sum_{\sigma \in S_{k+1}}(-1)^{i+k+|\sigma|}R_{\alpha_{\sigma(k+1)}}\cdots R_{\alpha_{\sigma(1)}}(d_{i}^{*}f).
\end{eqnarray*}
An easy calculation using Lemma \ref{lemma:R2} yields the following identities:
\begin{eqnarray*}
\sum_{\sigma\in S_{k+1}}(-1)^{|\sigma|+k}R_{\alpha_{\sigma(k+1)}}\cdots R_{\alpha_{\sigma(1)}}(d_0^{*}f) =
\sum_{i=1}^{k+1}(-1)^{i+1}L_{\rho(\alpha_i)}\Psi(f)(\alpha_1,\dots,\hat{\alpha}_i,\dots,\alpha_{(k+1)}),\\
\sum_{i=1}^{k}\sum_{\sigma \in S_{k+1}}(-1)^{i+k+|\sigma|}R_{\alpha_{\sigma(k+1)}}\cdots R_{\alpha_{\sigma(1)}}(d_{i}^{*}f) = \sum_{i<j}(-1)^{i+j}\Psi(f)([\alpha_{i},\alpha_{j}],\dots,\hat{\alpha}_i,\cdots,\hat{\alpha_j},\dots, \alpha_{k+1}).
\end{eqnarray*}
\\
These equations imply that $\Psi(\delta f)=d\Psi(f)$ holds.
\end{proof}

\subsection{Differentiating representations up to homotopy}\label{ss:differentiating_reps}

Here we will show that the structure operators for a representation up to homotopy of a Lie groupoid can
be differentiated to obtain a representation up to homotopy of the algebroid.

\begin{lemma}\label{lemma:hR}
Given $\alpha \in \Gamma(A)$ there is an operator
\begin{equation*}
\hat{R}_\alpha: \hat{C}^k_G(\End(E) )\rightarrow \hat{C}^{k-1}_G(\End(E))
\end{equation*}
defined by the formula
\begin{eqnarray*}
\hR_{\alpha}(F)(g_1,\dots,g_{k-1})(v):=\left. \frac{d}{d\epsilon}\right|_{\epsilon=0}F(g_1,\dots,g_{k-1},(\Phi^{\alpha}_{\epsilon}(x))^{-1}) (\gamma(t(\Phi^{\alpha}_{\epsilon}(x)))),
\end{eqnarray*}
where $x=s(g_{k-1})$, $v \in E_x$ and $\gamma \in \Gamma(E)$ is any section such that $\gamma(x)=v$. 
Moreover, for any  $h\in C^{\infty}(M)$ the following identities hold:
\begin{itemize}
\item[a)] $\hR_{h\alpha}(F)=s^*(h) \hR_{\alpha}(F)$,
\item[b)] $\hR_{\alpha}(s^{*}(h) F) = s^*(h) \hR_{\alpha}(F)$.
\end{itemize}
\end{lemma}
\begin{proof}
In order to prove that the operator is well defined we need to show that it is independent of the choice of section
$\gamma \in \Gamma(E)$. It is enough to prove that one can replace $\gamma$ by $\gamma ' =f \gamma$ where $f \in C^{\infty}(M)$
is any function with $f(x)=1$. For this we compute
\begin{eqnarray*}
\left. \frac{d}{d\epsilon} \right|_{\epsilon=0}F(g_1,&\dots&,g_{k-1},(\Phi^{\alpha}_{\epsilon}(x))^{-1}) (\gamma'(t(\Phi^{\alpha}_{\epsilon}(x))))=\\
&=&\left. \frac{d}{d\epsilon}\right|_{\epsilon=0}f(t(\Phi^{\alpha}_{\epsilon}(x)))F(g_1,\dots,g_{k-1},(\Phi^{\alpha}_{\epsilon}(x))^{-1}) ( \gamma(t(\Phi^{\alpha}_{\epsilon}(x))))\\
&=&\left. \frac{d}{d\epsilon}\right|_{\epsilon=0}f(x)F(g_1,\dots,g_{k-1},(\Phi^{\alpha}_{\epsilon}(x))^{-1}) ( \gamma(t(\Phi^{\alpha}_{\epsilon}(x))))\\
&&+\left. \frac{d}{d\epsilon}\right|_{\epsilon=0}f(t(\Phi^{\alpha}_{\epsilon}(x)))F(g_1,\dots,g_{k-1},x) ( \gamma(x))\\
&=&\left. \frac{d}{d\epsilon}\right|_{\epsilon=0}F(g_1,\dots,g_{k-1},(\Phi^{\alpha}_{\epsilon}(x))^{-1}) ( \gamma(t(\Phi^{\alpha}_{\epsilon}(x)))).\\
\end{eqnarray*}
Note that in the last step we used that $F$ is normalized. Part $a)$ is a consequence of the fact that $\Phi^{h\alpha}_{\epsilon}(x)=\Phi^{\alpha}_{h(x)\epsilon}(x)$. The last claim follows from a simple computation using again that $F$ is normalized.
\end{proof}

We establish some further properties of the maps $\hR_{\alpha}$, which will be useful for later computations.

\begin{lemma}\label{lemma:hR2}
For $F\in \hat{C}_{G}^{k}(\End^{m}(E))$ and $F'\in \hat{C}_{G}^{k'}(\End^{m'}(E))$ the following identities hold:
\begin{itemize}
\item[a)] $\hR_{\alpha}(F\circ F')=F \circ (\hR_{\alpha}F')$ if $k'>0$,
\item[b)] $\hR_{\alpha}(F\circ F')=(\hR_{\alpha}F) \circ F'$ if $k'=0$,
\item[c)] $\hR_{\alpha}(d_j^{*}F)=d_j^{*}(\hR_{\alpha}F)$ if $0<j<k$,
\item[d)] $\hR_{[\alpha,\beta]}F=\hR_{\alpha}\left(\hR_{\beta} (d_{k}^{*}F)\right) - \hR_{\beta}\left(\hR_{\alpha}(d_{k}^{*}F) \right)$.
\end{itemize}
\end{lemma}

\begin{proof}
All the statements are established in a straightforward manner except for the last one. 
The identity $d)$ follows form the fact that the Lie derivative of a bracket is the bracket of the Lie derivatives, as in the proof of Lemma \ref{lemma:R2}.
\end{proof}

\begin{definition}\label{def:differentiation_end_valued_things}
The map $\hPsi: \hat{C}_G^k(\End^m(E)) \to \Omega^k(A,\End^m(E))$ is defined by
\begin{eqnarray*}
\hPsi(F)(\alpha_1,\dots,\alpha_k):= (-1)^{km} \sum_{\sigma \in S_k}(-1)^{\sigma}\hR_{\alpha_{\sigma(k)}} \cdots \hR_{\alpha_{\sigma(1)}}F.
\end{eqnarray*}
\end{definition}

We observe that this map is well defined in view of Lemma \ref{lemma:hR}. 
\begin{definition}
For any graded vector bundle $E$, we will denote by $\Pi(A,E)$ the space of all $A$-connections on $E$ that
respect the degree. There is a map
\begin{equation*}
\overline{\Psi}: \overline{C}^1_G(\End^0(E)) \rightarrow \Pi(A,E)
\end{equation*}
given by the formula
\begin{equation*}
(\overline{\Psi}{F})_\alpha(\gamma):=\left. \frac{d}{d\epsilon}\right|_{\epsilon=0}F((\Phi^{\alpha}_{\epsilon}(x))^{-1}) ( \gamma(t(\Phi^{\alpha}_{\epsilon}(x)))),
\end{equation*}
for $\alpha \in \Gamma(A)$ and $\gamma \in \Gamma(E)$.
\end{definition}

\begin{lemma}\label{corollary:hPhi}
Suppose that $F \in \hat{C}^k_G(\End^m(E))$, $F' \in \hat{C}^{k'}_G(\End^{m'}(E))$, $F_1 \in  \overline{C}^1_G(\End^0(E))$ and $\nabla=\overline{\Psi}(F_1)\in \Pi(A,E)$.
Then, the sum
\begin{equation*}
-F_1\ \circ F+ \sum_{j=1}^{k}(-1)^{j+1}(d_{j}^{*}F) + (-1)^{k} F \circ F_1
\end{equation*}
belongs to the normalized subspace $ \hat{C}^{k+1}_G(\End^{m}(E))$. Moreover, the following equations hold:
\begin{itemize}
\item[a)] $\hPsi(F\circ F')=(-1)^{k(k'+m')}\hPsi(F) \wedge \hPsi(F')$,
\item[b)]$\hat{\Psi}\big(-F_1\ \circ F+ \sum_{j=1}^{k}(-1)^{j+1}(d_{j}^{*}F) + (-1)^{k} F \circ F_1\big)=(-1)^{k+m+1}d_{\nabla}(\hat{\Psi}(F)).$
\end{itemize}

\end{lemma}

\begin{proof}
The first claim can be checked by inspection. Part $a)$ follows from Lemma \ref{lemma:hR2}. Concerning part $b)$, one computes
\begin{eqnarray*}
\vspace{-1cm} d_{\nabla}(\hPsi(F))(\alpha_1,\dots,\alpha_{(k+1)})&=&\underbrace{\sum_{j=1}^{k+1}(-1)^{(j+1)}\nabla_{\alpha_j}\circ \hPsi(F)(\alpha_1,\dots,\hat{\alpha_j},\dots,\alpha_{k+1})}_A+\\
&&\underbrace{ \sum_{i<j}(-1)^{i+j}\hPsi(F)([\alpha_i,\alpha_j],\dots,\hat{\alpha}_i,\dots,\hat{\alpha_j},\dots,\alpha_{k+1})}_B +\\
&&\underbrace{ \sum_{j=1}^{k+1}(-1)^{j}\hPsi(F)(\alpha_1,\dots,\hat{\alpha}_j,\dots,\alpha_{k+1}) \circ \nabla_{\alpha_j}}_C.
\end{eqnarray*}
On the other hand, differentiating $$-F_1\ \circ F+ \sum_{j=1}^{k}(-1)^{j+1}(d_{j}^{*}F) + (-1)^{k} F \circ F_1,$$ and evaluating on $\alpha_1,\dots,\alpha_{k+1}$, we observe
that the part coming from $-F_1\ \circ F$ gives $(-1)^{k+m+1}A$, the part coming from $ \sum_{j=1}^{k}(-1)^{j+1}(d_{j}^{*}F)$ gives  $(-1)^{k+m+1}B$ and the part
coming from $ (-1)^{k} F \circ F_1$ gives  $(-1)^{k+m+1}C$.
\end{proof}
\begin{corollary}\label{Hom}
Let $E$ and $E'$ be graded vector bundles over $M$ and suppose that we are given $F_1\in  \overline{C}^1_G(\End^0(E))$ and $F'_1 \in  \overline{C}^1_G(\End^0(E'))$.
Then there is a map
\begin{equation*}
\hat{\Psi}: \hat{C}_G(\Hom(E,E')) \rightarrow \Omega(A, \Hom(E,E'))
\end{equation*}
given by the same formulas as in Definition \ref{def:differentiation_end_valued_things}. If we denote by $\nabla$ the $A$-connection induced on $E \oplus E'$
by differentiating $F_1$ and $F'_1$, then the map $\hat{\Psi}$ satisfies the equations
 \begin{itemize}
\item[a)] $\hPsi(T \circ F)=(-1)^{a(k+m)}\hPsi(T) \wedge \hPsi(F')$,
\item[b)] $\hPsi(F'\circ T)=(-1)^{k'(a+b)}\hPsi(F') \wedge \hPsi(T)$,
\item[c)]$\hat{\Psi}\big(-F'_1\ \circ T+ \sum_{j=1}^{k}(-1)^{j+1}(d_{j}^{*}T) + (-1)^{k} T \circ F_1\big)=(-1)^{k+m+1}d_{\nabla}(\hat{\Psi}(T)),$
\end{itemize}
for $F \in  \hat{C}^k_G(\End^m( E))$, $F' \in  \hat{C}^{k'}_G(\End^{m'}( E))$ and $T \in  \hat{C}^a_G(\Hom^b(E,E'))$
\end{corollary}
\begin{proof}
This follows from applying Lemma \ref{corollary:hPhi} to the vector bundle $E\oplus E'$ and observing that  $\hat{C}_G(\Hom(E,E')) \subset  \hat{C}_G(\End(E' \oplus E'))$.
\end{proof}
The following lemma establishes the compatibility between $\hPsi$ and $\Psi$, the proof is a straightforward computation.

\begin{lemma}\label{lemma:R_and_hR}
For $F\in \hat{C}_{G}^{k}(\End^{m}(E))$ with $(k,m)\neq(1,0)$ and $\eta \in \hat{C}^{p}(G,E^{q})$, the following identities hold:
\begin{itemize}
\item[a)] $R_{\alpha}(\tilde{F}(\eta))=(-1)^{k}\tilde{F}(R_{\alpha}\eta)$ if $p>0$,
\item[b)] $R_{\alpha}(\tilde{F}(\eta))=(-1)^{q}\widetilde{(\hR_{\alpha}(F))}\left(\eta\right)$ if $p=0$,
\item[c)] $\Psi\left(\tilde{F}(\eta)\right) = \hPsi(F) \wedge \Psi(\eta).$
\end{itemize}
\end{lemma}

We can now prove the main result of this section.

\begin{theorem}\label{theorem:differentiation_reps_up_to_homotopy}
Let $G$ be a Lie groupoid over $M$ with Lie algebroid $A$. 
There is a functor
\[\Psi: \URRep(G) \rightarrow  \RRep(A),\]
defined as follows: If $E\in \URRep(G)$ has structure operator
\[D=\tilde{F_0}+\tilde{ F_1}+\tilde{ F_2}+ \dots ,\]
then $\Psi(E)$ is the representation up to homotopy of $A$ on the graded vector bundle $E$ with structure operator
\begin{equation}\label{equationPsi(E)}
\Psi(D)= \hPsi(F_0)+\overline{\Psi}(F_1)+\hPsi(F_2)+\dots.
\end{equation}
Given a morphism 
\[\phi=\tilde{\phi_0}+\tilde{\phi_1}+\tilde{\phi_2}+\dots\] 
between representations up to homotopy of $G$, the operator
\[\Psi(\phi)=\hPsi(\phi_0)+\hPsi(\phi_1)+\hPsi(\phi_2)+\dots\] 
is a morphism between the corresponding representations up to homotopy of $A$. Moreover, for any $E \in \URRep(G)$, the linear map
\begin{equation}\label{equationPsi(phi)}
\Psi:\hat{C}(G,E) \rightarrow \Omega(A, \Psi(E)), 
\end{equation}
introduced in Definition \ref{def:differentiation_cochains}, is a morphism of chain complexes and is compatible with morphisms in the sense that
\begin{equation}\label{compatibility}
 \Psi(\phi(\eta))=\hPsi(\phi) (\Psi(\eta)).
\end{equation}
\end{theorem}

\begin{proof}
We  will denote the $A$-connection $\hPsi(F_1)$ by $\nabla$.
First, we need to prove that the operator $\Psi(D)$ defined in equation $(\ref{equationPsi(E)})$ equips the graded vector bundle $E$ with
the structure of a representation up 
to homotopy of $A$. We know that the operators $F_i$ satisfy the equations
\begin{eqnarray*}
 \sum_{j=0}^{k}(-1)^{j}\left(F_{j}\circ F_{k-j}\right) +  \sum_{j=1}^{k-1}(-1)^{j+1}d_{j}^{*}(F_{k-1})=0 ,
\end{eqnarray*}
which are the structure equations for representations up to homotopy of Lie groupoids as explained in Proposition \ref{decomposition with respect to F}. Applying $\hPsi$ to these equalities and using Lemma \ref{corollary:hPhi}, one obtains the equations
\[ d_{\nabla}(\hPsi(F_{k-1}))+ \sum_{i\notin \{1,k-1\}}\hPsi(F_{i}) \wedge \hPsi(F_{k-i})=0.\]
These are the structure equations for a representation up to homotopy of $A$ given in Proposition \ref{Rep-long}. Next, we need to prove that the map $\Psi(\phi)$ defined
in equation (\ref{equationPsi(phi)}) is a morphism of representations up to homotopy of $A$. Since $\phi$ is a morphism from $E$ to $E'$, the following identities hold:
\begin{eqnarray}
\sum_{i+j=k}(-1)^{j+1}\phi_j \circ 
F_{i}+\sum_{i+j=k}F_{j}^{'}
\circ \phi_{i}+\sum_{j=1}^{k-1}(-1)^{j}d^*_j(\phi_{k-1})=0.
\end{eqnarray}
Applying $\hPsi$ to this equality and using Corollary \ref{Hom}, we arrive at
\[d_{\nabla}(\hPsi(\phi_{k-1})+\sum_{j\neq 1}[\hPsi(F_j),\hPsi(\phi_{k-j})]=0,\]
which are the structure equations for morphisms at the infinitesimal level. A simple computation shows that this construction respects the composition. We conclude that
the functor $\Psi$ is well defined.

Let us now prove that the map $\Psi:\hat{C}(G,E) \rightarrow \Omega(A,\Psi(E))$ is a morphism of chain complexes. 
We observe that Lemma \ref{lemma:R_and_hR} implies that the diagram

\begin{eqnarray*}
\xymatrix{
\hat{C}(G,E) \ar[r]^{\tilde{F}_k}\ar[d]_{\Psi} & \hat{C}(G,E) \ar[d]^{\Psi}\\
\Omega(A,E) \ar[r]_{\hPsi(F_k)} & \Omega(A,E)
}
\end{eqnarray*}
commutes for $k\ne 1$. We still need to prove that the diagram
\begin{eqnarray*}
\xymatrix{
\hat{C}(G,E) \ar[r]^{\tilde{F_1}}\ar[d]_{\Psi} & \hat{C}(G,E) \ar[d]^{\Psi}\\
\Omega(A,E) \ar[r]_{d_{\nabla}} & \Omega(A,E)
}
\end{eqnarray*}
commutes. This is a formal consequence of Corollary \ref{Hom} $c)$, applied to the case of maps from the trivial vector bundle $\mathbb{R}$ with
connection given by the constant section $1 \in \overline{C}^1_G(\End^0(\mathbb{R}))$ to the vector bundle $E$ with connection given by $F_1 \in  \overline{C}^1_G(\End^0(E))$. 

The last claim follows immediately from part $c)$ of Lemma \ref{lemma:R_and_hR}.
\end{proof}

We will show now that the functor $\Psi$ constructed above behaves in a natural way with respect to the operation of
pulling back representations up to homotopy. Given a morphism of Lie groupoids
\begin{eqnarray*}
\xymatrix{
G \ar[r]^{\gamma}\ar@<-1mm>[d]_{s}\ar@<1mm>[d]^{t} & Q \ar@<-1mm>[d]_{s} \ar@<1mm>[d]^{t}\\
M \ar[r]^{\gamma}& N,
}
\end{eqnarray*}
one can differentiate it to obtain a morphism between the corresponding Lie algebroids
\begin{eqnarray*}
\xymatrix{
A \ar[r]^{\gamma}\ar[d] & B \ar[d]\\
M \ar[r]^{\gamma}& N.
}
\end{eqnarray*}

Assume that $(C(Q,E),D)$ is a representation up to homotopy of $Q$ on a graded vector bundle $E\to N$.
Let $(F_k)_{k\geq 0}$ be the structure maps corresponding to $D$. We define $\gamma^{*}(F_k) \in C_G(\End(\gamma^*(E)))$ by the formula
\begin{eqnarray*}
(\gamma^{*}F_k)(g_1,\dots,g_k):=F_k(\gamma(g_1),\dots,\gamma(g_k)).
\end{eqnarray*}
Clearly this sequence equips the vector bundle $\gamma^*(E)$ with the structure of a representation up homotopy of $G$.
Observe that the pull back actually extends to a functor $\URRep(Q) \to \URRep(G)$.

Furthermore, there is a chain map given by the formula
\begin{eqnarray*}
&\gamma^{*}: (C(Q,E),D) \to (C(G,\gamma^{*}E),\gamma^{*}D),\\
&(\gamma^{*}\eta)(g_1,\dots,g_k):=\eta(\gamma(g_1),\dots,\gamma(g_k)).
\end{eqnarray*}

One can also pull back representations up to homotopy along morphisms of Lie algebroids in a functorial way,
and in this situation the pull back of cochains also yields
a map of complexes
\begin{eqnarray*}
&\hat{\gamma}^{*}: (\Omega(B,E),D) \to (\Omega(A,\gamma^{*}E),\gamma^{*}D).
\end{eqnarray*}

All this operations are compatible in the following sense:
\begin{remark}\label{lemma:pullback_naturality}
Let $\gamma: G \to Q$ be a morphism of Lie groupoids and let $\gamma: A \to B$ be the induced morphism of Lie algebroids. 
Given a representation up to homotopy $(C(Q,E),D)$, all the faces of the following cube commute:
\begin{eqnarray*}
\xymatrix{
 & \Omega(B,E) \ar '[d]^{\hPsi(D)} [dd] \ar[rr]^{\gamma^{*}} & & \Omega(A,\gamma^{*}E)\ar[dd]^{\hPsi(\gamma^{*}D)} \\
C(Q,E) \ar[dd]_D \ar[rr]_<<<<<<<<<<{\gamma^{*}} \ar[ru]^{\Psi} & & C(G,\gamma^{*}E) \ar[dd]^<<<<<<<<{\gamma^{*}D} \ar[ru]_{\Psi} & \\
 & \Omega(B,E)[1] \ar '[r]_-{\gamma^{*}} [rr] & & \Omega(A,\gamma^{*}E)[1] \\
C(Q,E)[1]\ar[rr]_{\gamma^{*}} \ar[ru]^{\Psi} & & C(G,\gamma^{*}E)[1]\ar[ru]^{\Psi} & \\
}
\end{eqnarray*}
\end{remark}

\section{The van Est theorem and deformations}\label{section:deformations}

\subsection{An isomorphism theorem}\label{subsection:isomorphism}

We will now prove a van Est theorem for representations up to homotopy, which provides conditions
under which the differentiation map $\Psi:\hat{C}(G,E) \rightarrow \Omega(A,\Psi(E))$ induces an isomorphism 
in cohomology (in certain degrees).

Given a surjective submersion $\pi:M \rightarrow N$, there is a groupoid $M\times_N M$ over $M$ with target and source maps
given by the projections on the first and second component, respectively.

\begin{definition}
An elementary Lie groupoid $G$ is a Lie groupoid that is isomorphic to $M\times_N M$ for some surjective submersion $\pi: M \rightarrow N$ which
admits a section $\iota:N \rightarrow M$.
\end{definition}

\begin{lemma}\label{lemma:elementary}
Let $G$ be an elementary groupoid. Then, any unital representation up to homotopy $E$ of $G$ is quasi-isomorphic to a unital representation up to homotopy $E'$ with the property
that the structure operator $D$ of $E'$ is of the form:
\begin{equation*}
D=\tilde{F}'_0+\tilde{F}'_1,
\end{equation*}
with respect to the decomposition given in Proposition \ref{decomposition with respect to F}.
\end{lemma}

\begin{proof}
We know that $G$ is the groupoid associated to a surjective submersion $\pi: M \rightarrow N$ that admits a section $\iota: N \rightarrow M$. 
Let us denote by $\mathcal{N}$ the unit groupoid of $N$. Then, there are canonical morphisms of Lie groupoids $\pi:G \rightarrow \mathcal{N}$ and
$\iota: \mathcal{N} \rightarrow G$.
We claim that $E$ is quasi-isomorphic to $\pi^* (\iota^*(E))$.
Given an arrow $g \in G$, there is a unique arrow $\nu(g)\in G$ with the property that $s(\nu(g))=t(g)$ and $t(\nu(g))=t(\iota \pi(g))$. 
Let 
\[D=\tilde{F}_0+ \tilde{F}_1+ \tilde{F}_2+\cdots,\]
be the structure operator for $E$. We define a morphism of representations up to homotopy:
\[\phi: E \rightarrow \pi^* (\iota^*(E))\]
by 
\[\phi=\phi_0+\phi_1+\phi_2+\cdots,\]
where

\[ \phi_k(g_1,\dots.g_k)=F_{k+1}(\nu(g_1),g_1,\dots,g_k).\]
We denote the structure operator of $\pi^* (\iota^*(E))$ by $D'$ and observe that
it has the form:
\[D'=\tilde{F}'_0+ \tilde{F}'_1.\]
This is the case because $\pi^* (\iota^*(E))$ is the pull-back of a unital representation up to homotopy of the unit groupoid $\mathcal{N}$, and those are always
of this form. Moreover one immediately checks that:
\begin{eqnarray*}
\tilde{F}'_0(x)&=&F_0(\iota (\pi(x))),\\
\tilde{F}'_1(g) &=& \id.
\end{eqnarray*}
In order to prove that $\phi$ is a morphism of representations we need to establish the following equations:
\begin{eqnarray*}
\sum_{i+j=k}(-1)^{j}\phi_j(g_1,\dots,g_j) \circ
F_{i}(g_{j+1},\dots,g_k)&-& \sum_{i+j=k}F_{j}^{'}(g_1,\dots , g_j)
\circ \phi_{i}(g_{j+1},\dots,g_k)\\
&+&\sum_{j=1}^{k-1}(-1)^{j+1}\phi_{k-1}(g_1,\dots,g_jg_{j+1},\dots
,g_k)=0.
\end{eqnarray*}
For this we compute directly:
\begin{eqnarray*}
&& \hspace{-0.5cm} \sum_{i+j=k}(-1)^{j}\phi_j(g_1,\dots,g_j) \circ
F_{i}(g_{j+1},\dots,g_k)- \sum_{i+j=k}F_{j}^{'}(g_1,\dots , g_j)
\circ \phi_{i}(g_{j+1},\dots,g_k)\\
&&+\sum_{j=1}^{k-1}(-1)^{j+1}\phi_{k-1}(g_1,\dots,g_jg_{j+1},\dots
,g_k)\\
&=&\sum_{i+j=k}(-1)^{j}F_{j+1}(\nu(g_1),g_1,\dots,g_j) \circ
F_{i}(g_{j+1},\dots,g_k)- F'_0 \circ F_{k+1}(\nu(g_1),g_1,\dots, g_k)\\
&&-F'_1(g_1) \circ F_{k}(\nu(g_2),g_2,\dots, g_k)
+\sum_{j=1}^{k-1}(-1)^{j+1}F_{k}(\nu(g_1),g_1,\dots,g_jg_{j+1},\dots
,g_k)\\
&=&\sum_{i+j=k+1}(-1)^{j+1}F_{j}(\nu(g_1),g_1,\dots,g_j) \circ
F_{i}(g_{j+1},\dots,g_k)
- F_{k}(\nu(g_2),g_2,\dots, g_k)\\
&&+\sum_{j=1}^{k-1}(-1)^{j+1}F_{k}(\nu(g_1),g_1,\dots,g_jg_{j+1},\dots
,g_k)\\
&=&\sum_{i+j=k+1}(-1)^{j+1}F_{j}(\nu(g_1),g_1,\dots,g_j) \circ
F_{i}(g_{j+1},\dots,g_k)
- F_{k}(\nu(g_1)g_1,g_2,\dots, g_k)\\
&&+\sum_{j=1}^{k-1}(-1)^{j+1}F_{k}(\nu(g_1),g_1,\dots,g_jg_{j+1},\dots
,g_k)\\
&=& 0.
\end{eqnarray*}
The last equation holds because of the structure equations for the operator $D$.
We conclude that $\phi$ is a morphism of representations up to homotopy. One also
checks immediately that $\phi_0$ induces isomorphism in cohomology at every point. This completes the proof.
\end{proof}

\begin{proposition}\label{easy}
Let $G$ be an elementary Lie groupoid with $l$-connected source fibers and suppose that $E=\bigoplus_{k=a}^bE^k$ is a unital representation up to homotopy of $G$.
Then, the map induced in cohomology 
\begin{equation*}
\Psi:H^n(G,E) \rightarrow H^n(A,\Psi(E)),
\end{equation*}
is an isomorphism for $a \leq n\leq a+l$.
\end{proposition}
\begin{proof}
We can obviously assume that $a=0$.
In view of Lemma \ref{lemma:elementary} and the fact that quasi-isomorphisms induce isomorphisms in cohomology, 
it suffices to consider the case where the structure operator of $E$ is of the form:
\begin{equation*}
D=\tilde{F}_0+\tilde{F}_1.
\end{equation*}
This means that the cohomology $H(G,E)$ is computed by the double complex:

\begin{eqnarray*}
\xymatrix{
  \vdots &\vdots &\vdots  & \\
\hat{C}^2(G,E^0) \ar[r]^-{{\tilde{F}_0}}\ar[u]^{\tilde{F}_1} &
\hat{C}^2(G,E^1)
\ar[r]^-{{\tilde{F}_0}}\ar[u]^{\tilde{F}_1} & \hat{C}^2(G,E^2)\ar[r]^-{{\tilde{F}_0}}\ar[u]^{\tilde{F}_1}& \dots\\
\hat{C}^1(G,E^0)\ar[r]^-{{\tilde{F}_0}}\ar[u]^{\tilde{F}_1} &
\hat{C}^1(G,E^1)
\ar[r]^-{{\tilde{F}_0}}\ar[u]^{\tilde{F}_1} & \hat{C}^1(G,E^2)\ar[r]^-{{\tilde{F}_0}}\ar[u]^{\tilde{F}_1}& \dots\\
\hat{C}^0(G,E^0) \ar[r]^-{{\tilde{F}_0}}\ar[u]^{\tilde{F}_1} &
\hat{C}^0(G,E^1)
\ar[r]^-{{\tilde{F}_0}}\ar[u]^{\tilde{F}_1} & \hat{C}^0(G,E^2)\ar[r]^-{{\tilde{F}_0}}\ar[u]^{\tilde{F}_1}& \dots }
\end{eqnarray*}

Here the vertical lines are given by the complexes associated to ordinary representations of $G$.
The differentiation operator $\Psi$ induces a map of double complexes to the  corresponding double complex
associated to the Lie algebroid:
\begin{eqnarray*}
\xymatrix{
 \vdots &\vdots &\vdots  & \\
\Omega^2(A,E^0) \ar[r]^-{{\tilde{F}_0}}\ar[u]^{\tilde{F}_1} &
\Omega^2(A,E^1)
\ar[r]^-{{\tilde{F}_0}}\ar[u]^{\tilde{F}_1} & \Omega^2(A,E^2)\ar[r]^-{{\tilde{F}_0}}\ar[u]^{\tilde{F}_1}& \dots\\
\Omega^1(A,E^0)\ar[r]^-{{\tilde{F}_0}}\ar[u]^{\tilde{F}_1} &
\Omega^1(A,E^1)
\ar[r]^-{{\tilde{F}_0}}\ar[u]^{\tilde{F}_1} & \Omega^1(A,E^2)\ar[r]^-{{\tilde{F}_0}}\ar[u]^{\tilde{F}_1}& \dots\\
\Omega^0(A,E^0) \ar[r]^-{{\tilde{F}_0}}\ar[u]^{\tilde{F}_1} &
\Omega^0(A,E^1)
\ar[r]^-{{\tilde{F}_0}}\ar[u]^{\tilde{F}_1} & \Omega^0(A,E^2)\ar[r]^-{{\tilde{F}_0}}\ar[u]^{\tilde{F}_1}& \dots }
\end{eqnarray*}
Now, the van Est isomorphism for ordinary representations, Theorem $4$ of \cite{Crainic}, guarantees that this map induces an
isomorphism in the vertical cohomology up to degree $l$. A standard spectral sequence argument implies that the map
also induces isomorphisms in the total cohomology up to degree $l$.

\end{proof}

A left action of a Lie groupoid $G$ on a manifold $P \stackrel{\nu}{\rightarrow} M$ over
$M$ is a map $G_1\times_{M} P \rightarrow P$ defined on the space $G_1\times_{M} P$ of
pairs $(g, p)$ with $s(g)= \nu(p)$, which satisfies $\nu(gp)= t(g)$ and is compatible with the 
units and the composition in $G$. Given an action of $G$ on
$P\stackrel{\nu}{\rightarrow} M$ there is an action groupoid, denoted by $G\ltimes P$. The base of this groupoid is $P$, 
the space of arrows is $G_1\times_{M}P$, the source map is the second projection, while the target map
is the action. The multiplication in this groupoid is $(g, p)(h, q)= (gh, q)$.

For all $k\geq 0$ a Lie groupoid $G$ acts on the manifold $G_{k+1} \stackrel{t}{\rightarrow }M$ with action given by the formula
\[ g  (g_1,\dots,g_{k+1})= (gg_1,\dots,g_{k+1}). \]
We will denote the action groupoid associated to this action by $G^{(k)}$ and set $G^{(-1)}=G$. Observe that $G^{(k)}$ is isomorphic to the groupoid associated to 
the surjective submersion $ d_0:G_{k+1} \rightarrow G_{k}$, where we use the simplicial face maps introduced in equation (\ref{simplicial faces}). 
In particular, $G^{(k)}$ is an elementary groupoid. There are natural diffeomorphisms

\[G_0^{(k)}\cong G_{k+1}, \,G_1^{(k)} \cong G_{k+2},\] 
and the source and target maps correspond to $d_0$ and $d_1$ respectively. We will use these identifications freely from now on.
For $i=0,\dots, k,$ there are morphisms of Lie groupoids $\flat_i:G^{(k)}\rightarrow G^{(k-1)}$ defined as follows: At the level of morphisms, $\flat_i$ is given by
$d_{i+2}:G_{k+2} \rightarrow G_{k+1}$ and at the level of objects it is $d_{i+1}:G_{k+1}\rightarrow G_k$. 
\begin{lemma}\label{lemma:simplicial}
The following statements hold:
\begin{enumerate}
\item The sequence $(G^{(k)})_{k\geq 0}$ together with the morphisms $\flat_i$ form a semisimplicial groupoid
\[  
\xymatrix{
\cdots G^{(2)} \ar@ <+9 pt>[r]^{\flat_0} \ar@ <0 pt>[r]^{\flat_1}\ar@ <-9 pt>[r]_{\flat_2} &    G^{(1)}\ar@ <+4 pt>[r]^{\flat_0} \ar@ <-4 pt>[r]_{\flat_1} & G^{(0)}.
}  
\]
\item Let $\pi: G^{(k)} \rightarrow G$ be the morphism of Lie groupoids defined on arrows by $(g_1,\dots,g_{k+2}) \mapsto g_{1}$, and on objects by $(g_1,\dots,g_{k+1}) \mapsto t(g_1)$.
Any morphism $\gamma:G^{(k)} \rightarrow G$ which is obtained by composing the morphisms $\flat_i$ is equal to $\pi$.

\end{enumerate}
\end{lemma}
\begin{proof}
The first claim is a formal consequence of the fact that the operators $d_i$ give the sequence $(G_k)_{k\geq 0}$ the structure of a semisimplicial manifold.
For the second part we observe that $\flat_i:G^{(k)} \rightarrow G^{(k-1)}$ is given on arrows by $d_{i+2}:G_{k+2}\rightarrow G_{k+1}$, and therefore
it does not change the first component.
\end{proof}

\begin{remark}
In order to simplify the notation, we will write $E$ for the representation up to homotopy $\pi^*(E)$ of $G^{(k)}$. The morphism $\flat_i$ induces a pullback map:
\[ \flat_i^*:\hat{C}(G^{(k-1)},E) \rightarrow \hat{C}(G^{(k)},\flat_i^*(E)).\]  In view of Lemma \ref{lemma:simplicial}, we know that $\flat_i^*(E)=E$, so we can write:
\[ \flat_i^*:\hat{C}(G^{(k-1)},E) \rightarrow \hat{C}(G^{(k)},E).\] 
\end{remark}

We will need the following fact.

\begin{lemma} \label{lemma:exact sequence}
 Let   $\flat^*:\hat{C}(G^{(k-1)},E) \rightarrow \hat{C}(G^{(k)},E)$ be the linear map
\[\flat^*:=\sum_{i=0}^k(-1)^i{ \flat_i}^*.\]
Then, the sequence of maps
\begin{equation}\label{firstexact}
 0 \rightarrow \hat{C}(G,E)^{n} \stackrel{\flat^*}{\rightarrow}
\hat{C}(G^{(0)},E)^{n}
 \stackrel{\flat^*}{\rightarrow} \hat{C}(G^{(1)},E)^{n} \stackrel{\flat^*}{\rightarrow}\hat{C}(G^{(2)},E)^{n} \stackrel{\flat^*}{\rightarrow} \cdots, 
 \end{equation}
 is a long exact sequence. Moreover, applying the functor $\Psi$, one obtains a long exact sequence:
\begin{equation}\label{secondexact}
0 \rightarrow \Omega(A,E)^{n} \stackrel{\flat^*}{\rightarrow}   \Omega(A^{(0)},E)^{n}
 \stackrel{\flat^*}{\rightarrow} \Omega(A^{(1)},E)^{n} \stackrel{\flat^*}{\rightarrow} \Omega(A^{(2)},E)^{n} \stackrel{\flat^*}{\rightarrow} \cdots
 \end{equation}
where $A^{(k)}$ denotes the Lie algebroid of $G^{(k)}$.
\end{lemma}

\begin{proof}
Let us consider the first claim. We observe that there are natural diffeomorphisms:
\begin{eqnarray*}
\phi:G_m^{(k)}&\rightarrow& G_{k+m+1},\\
((g_1, h_1^1,\dots,h_{k+1}^1), \dots, (g_m, h_1^m,\dots,h_{k+1}^m))&\mapsto &(g_1,\dots,g_m,h_1^m,\dots,h_{k+1}^m).
\end{eqnarray*}
These diffeomorphisms induce isomorphisms of vector spaces:
\[\phi^*:C^{k+m+1}(G,E) \rightarrow C^m(G^{(k)},E).\]
Let us denote by $C_{m}^{k+m+1}(G,E)$ the subspace
\[C_{m}^{k+m+1}(G,E):= \left\{\eta \in C^{k+m+1}(G,E): \eta(g_1,\dots, g_i, 1 ,g_{i+2},\dots, g_{k+m+1})=0, \text{ for } i<m\right\}.\]
Then the isomorphism $\phi^*$ restricts to an isomorphism:
\[\phi^*:C_m^{k+m+1}(G,E) \rightarrow \hat{C}^m(G^{(k)},E).\]
Moreover, the diagram:
\begin{eqnarray*}
\xymatrix{
C_{m}^{k+m}(G,E) \ar[r]^{d_{i+m+1}^*}\ar[d]^{\phi^*} & C_{m}^{k+m+1}(G,E)\ar[d]^{\phi^*}\\
 C^m(G^{(k-1)},E) \ar[r]^{\flat_i^*} & C^m(G^{(k)},E) }
\end{eqnarray*}
commutes.
Thus, if we denote by $b^*:C_m^{k+m}(G,E)\rightarrow C_m^{k+m+1}(G,E)$ the map
\[b^*:=\sum_{i=0}^k(-1)^id_{i+m+1}^*,\]
it is sufficient to prove that the sequence
\begin{equation}\label{complex}
 0 \rightarrow C_m^m(G,E) \stackrel{b^*}{\rightarrow}
C_m^{m+1}(G,E)
 \stackrel{\flat^*}{\rightarrow} C_m^{m+2}(G,E) \stackrel{b^*}{\rightarrow}C_m^{m+3}(G,E)\stackrel{b^*}{\rightarrow} \cdots, 
 \end{equation}
is exact. Let $s^*: C_m^{m+k+1}(G,E) \rightarrow C_m^{m+k}(G,E)$ be defined by the formula:
\begin{equation*}
s^*(\eta)(g_1,\dots,g_{k+m})=\eta(g_1,\dots, g_m, 1,g_{m+1},\dots,g_{m+k}).
\end{equation*}
Then, one easily checks that:
\begin{eqnarray*}
 s^*d_{i+m+1}^*=
\begin{cases}
\id & \textrm{for } i=0,\\
 d_{i+m}^* s^*& \textrm{for } i>0.
\end{cases}
\end{eqnarray*}
These relations imply that $b^* s^*+s^*b^*=\id$. We conclude that the sequence (\ref{complex}) is exact.
Since all the operators in this proof are also present in the case of the sequence (\ref{secondexact}), the same 
proof applies.
\end{proof}

\begin{theorem}\label{isomorphism}
Let $G$ be a Lie groupoid with $l$-connected source fibers and suppose that \[E=\bigoplus_{k=a}^bE^k\] is a unital representation up to homotopy of $G$.
Then, the map induced in cohomology
\begin{equation*}
\Psi:H^n(G,E) \rightarrow H^n(A,\Psi(E)),
\end{equation*}
is an isomorphism in degrees $a \leq n\leq a+l$.
\end{theorem}
\begin{proof}
Clearly, we can assume that $a=0$.
Let us denote by $\flat^*:\hat{C}(G^{(k-1)},E) \rightarrow \hat{C}(G^{(k)},E)$ the linear map
\[\flat^*:=\sum_{i=0}^k(-1)^i \flat_i^*.\]

One can organize all these linear maps in a double complex as follows:
\begin{eqnarray*}
\xymatrix{
 & \vdots &\vdots &\vdots  & \\
0\ar[r]& \hat{C}(G,E)^{2} \ar[r]^-{{\flat}^*}\ar[u]^D &
\hat{C}(G^{(0)},E)^{2}
\ar[r]^-{\flat^*}\ar[u]^D & \hat{C}(G^{(1)},E)^{2}\ar[r]^-{\flat^*}\ar[u]^D& \dots\\
0\ar[r]&\hat{C}(G,E)^{1} \ar[r]^-{{\flat}^*}\ar[u]^D &
\hat{C}(G^{(0)},E)^{1}
\ar[r]^-{\flat^*}\ar[u]^D & \hat{C}(G^{(1)},E)^{1}\ar[r]^-{\flat^*}\ar[u]^D& \dots\\
0\ar[r]&\hat{C}(G,E)^{0} \ar[r]^-{{\flat}^*}\ar[u]^D &
\hat{C}(G^{(0)},E)^{0}
\ar[r]^-{\flat^*}\ar[u]^D & \hat{C}(G^{(1)},E)^{0}\ar[r]^-{\flat^*}\ar[u]^D& \dots }
\end{eqnarray*}

Here, the vertical operator $D$ is the differential corresponding to the representation up to homotopy $E$.
Note that $D$ commutes with $\flat^*$ because it commutes with each $\flat_i^*$ by Remark \ref{lemma:pullback_naturality}.
By Lemma \ref{lemma:exact sequence} we know that the rows of this double complex are acyclic. 
We denote by $\mathrm{Tot(G)}$ the total complex:
\[ \mathrm{Tot}^n(G)=\bigoplus_{k \geq 0} \hat{C}(G^{(k)},E)^{n-k}.\]

The general homological algebra argument
implies that the inclusion of the first column induces an isomorphism in cohomology:
\[ \flat^*: H(G, E) \rightarrow H(\mathrm{Tot}(G)). \]

By applying the functor $\Psi$ to the diagram above we obtain a double complex:

\begin{eqnarray*}
\xymatrix{
 & \vdots &\vdots &\vdots  & \\
0\ar[r]& \Omega(A,E)^{2} \ar[r]^-{{\flat}^*}\ar[u]^D &
\Omega(A^{(0)},E)^{2}
\ar[r]^-{\flat^*}\ar[u]^D & \Omega(A^{(1)};E)^{2}\ar[r]^-{\flat^*}\ar[u]^D& \dots\\
0\ar[r]&\Omega(A,E)^{1} \ar[r]^-{{\flat}^*}\ar[u]^D &
\Omega(A^{(0)},E)^{1}
\ar[r]^-{\flat^*}\ar[u]^D & \Omega(A^{(1)},E)^{1}\ar[r]^-{\flat^*}\ar[u]^D& \dots\\
0\ar[r]&\Omega(A,E)^{0} \ar[r]^-{{\flat}^*}\ar[u]^D &
\Omega(A^{(0)},E)^{0}
\ar[r]^-{\flat^*}\ar[u]^D & \Omega(A^{(1)},E)^{0}\ar[r]^-{\flat^*}\ar[u]^D& \dots }
\end{eqnarray*}
Again, by Lemma \ref{lemma:exact sequence} the rows of this diagram are exact.
We denote by $\mathrm{Tot(A)}$ the total complex:
\[ \mathrm{Tot}^n(A)=\bigoplus_{k \geq 0}\Omega(A^{(k)},E)^{n-k}.\]

As before, the inclusion of the first column induces an isomorphism in cohomology:
\[ \flat^*: H(A, E) \rightarrow H(\mathrm{Tot}(A)). \]

We know that for $k\geq 0$ the groupoids $G^{(k)}$ are elementary. Moreover, the source fibers of $G^{(k)}$ are diffeomorphic
to those of $G$. Therefore,  Proposition \ref{easy} implies that for $k\geq 0$ the map:
\[\Psi:  \hat{C}(G^{(k)},E)^{n} \rightarrow \Omega(A^{(k)},E)^{n} \]
induces isomorphisms up to degree $l$. The usual spectral sequence argument then implies that

\[\Psi:  \mathrm{Tot}(G) \rightarrow \mathrm{Tot}(A) \]
induces isomorphisms up to degree $l$.
We now consider the commutative diagram
\begin{eqnarray*}
\xymatrix{
H^n(G,E) \ar[r]^{\flat^*}\ar[d]^{\Psi} &H^n(\mathrm{Tot}(G))\ar[d]^{\Psi}\\
H^n(A,E) \ar[r]^{\flat^*}&H^n(\mathrm{Tot}(A)). 
}
\end{eqnarray*}
For $n \leq l$ three of the maps above are isomorphism, so the fourth one also is. This completes the proof.

\end{proof}

 \subsection{Differentiating the adjoint and deformations}\label{ss:connection}
Here we show that by applying the differentiation functor $\Psi$ to the adjoint representation of a Lie groupoid, one obtains the 
adjoint representation of the Lie algebroid. We then use this fact together with Theorem \ref{isomorphism} to prove Conjecture $1$ of \cite{CrainicMoerdijk}.

As before,  $A$ is the Lie algebroid of $G$ and $\sigma$ is a connection
on $G$. That is, $\sigma$ is a splitting of the short exact sequence
\begin{equation*}
 0\rightarrow \ker(ds) \rightarrow TG \rightarrow s^*(TM) \rightarrow 0,
\end{equation*}
which coincides with the natural splitting over $M$.
The choice of the connection $\sigma$ induces operators, $\omega$, $\overline{\sigma}$ and $\overline{\omega}$ 
as explained in Section \ref{preliminaries}. Given a vector field $X$ on $M$, we will denote by  $\hat{X}$ the horizontal vector field on $G$ determined by $X$,
namely: 
\[\hat{X}(g)=\sigma(g)(X(s(g))).\] 
For a section $\alpha \in \Gamma(A)$, we will denote by $\hat{\alpha}$ the right invariant vector field on $G$ associated with $\alpha$. We also define the vector fields:
\begin{eqnarray*}
\check{X}(g):=d\iota (\hat{X}(g^{-1})),\\
\check{\alpha}(g):=d\iota (\hat{\alpha}(g^{-1})),
\end{eqnarray*}
where $\iota:G \rightarrow G$ is the inverse map.
\begin{remark}
The connection $\sigma$ on $G$ induces a connection $\nabla$ on $A$ by the formula:
\begin{equation}
\nabla_X(\alpha)(p)=[\hat{X},\hat{\alpha}](p),\, p \in M.
\end{equation}
\end{remark}

The vector fields $\hat{\alpha}$ and $\check{X}$ are projectable with respect to the target  map $t$ and they project onto the
vector fields $\rho(\alpha)$ and $X$, respectively. This implies the identity
\begin{eqnarray*}
[\rho(\alpha),X]=dt\left([\hat{\alpha},\check{X}]\right).
\end{eqnarray*}

\begin{lemma}
Let $ad_{\nabla}(A)$ be the adjoint representation of $A$ induced by $\nabla$, where $\nabla$ is determined by $\sigma$ as above.
Then the  $A$-connection part of $ad_{\nabla}(A)$ is given by the formulas:

 \begin{eqnarray*}
 \nabla^{bas}_{\alpha}(X)&=&[\rho(\alpha),X]-dt[\hat{\alpha},\hat{X}],\\ 
\nabla^{bas}_{\alpha}(\beta)&=&[\hat{\alpha},\hat{\beta}]-[\hat{\alpha},\widehat{\rho(\beta)}].
\end{eqnarray*}

\end{lemma}

\begin{proposition}\label{Ad=ad}
Let $G$ be a Lie groupoid equipped with a connection $\sigma$ and denote by $\nabla$
the induced connection on $A$. Then
\begin{equation*} 
\Psi(\Ad_{\sigma}(G))=\ad_{\nabla}(A).
\end{equation*}
\end{proposition}
\begin{proof}
Let us write the structure operator of $\Ad_{\sigma}(G)$ as:
\[D=\tilde{F_0} +\tilde{F_1}+\tilde{F_2}.\]
Clearly, $\hat{\Psi}(F_0)=\hat{\Psi}(\rho)=\rho$ is the coboundary operator in $\ad_{\nabla}(A)$.
Next, we need to prove that
\begin{equation}
\nabla^{bas}_{\alpha}(X)=\bar{\Psi}(F_1)_\alpha(X).
\end{equation}
Let us fix a point $p \in M$ and set $U=dt\left([\hat{\alpha},\check{X}]\right)(p)=[\rho(\alpha),X](p)$, $V=-dt\left([\hat{\alpha},\hat{X}]\right)(p)$ and $Z=\bar{\Psi}(F_1)_\alpha(X)(p)$;
we have seen that $\nabla_{\alpha}(X)= U+V$.
We need to prove that $U+V=Z$.
If we denote by $\Phi$ the right invariant flow associated to $\hat{\alpha}$, the flow associated to
the vector field $\rho(\alpha)$ is given by $t \Phi_{\epsilon}$. More explicitly:
\begin{eqnarray*}
U= \left.\frac{d}{d\epsilon}\right|_{\epsilon=0}Dt\circ D\Phi_{-\epsilon}(X(t\Phi_{\epsilon}(p))),\\
V=- \left. \frac{d}{d\epsilon}\right|_{\epsilon=0}Dt \circ D\Phi_{-\epsilon}( \sigma_{\Phi_{\epsilon}(p)}X(p)).
\end{eqnarray*}

On the other hand, using the property that
\begin{equation*}
(\Phi_{\epsilon}(p))^{-1}=\Phi_{-\epsilon}(t(\Phi_{\epsilon}(p)),
\end{equation*}
we compute:

\begin{eqnarray*}
 Z=\bar{\Psi}(F_1)_\alpha(X)(p)&=&\left. \frac{d}{d\epsilon}\right|_{\epsilon=0}F_1((\Phi_{\epsilon}(p))^{-1})(X(t(\Phi_{\epsilon}(p))))\\
&=&\left. \frac{d}{d\epsilon}\right|_{\epsilon=0}F_1(\Phi_{-\epsilon}(t(\Phi_{\epsilon}(p))))(X(t(\Phi_{\epsilon}(p))))\\
&=&\left. \frac{d}{d\epsilon}\right|_{\epsilon=0}Dt ( \sigma_{(\Phi_{-\epsilon}(t(\Phi_{\epsilon}(p))))}(X(t(\Phi_{\epsilon}(p))))).\\
\end{eqnarray*}
This expression can be written as:

\begin{equation*}
\left.\frac{d}{d\epsilon}\right|_{\epsilon=0}W(\epsilon, \epsilon),
\end{equation*} 
where
\begin{eqnarray*}
W(\epsilon,\zeta)=  Dt\circ D\Phi_{-\epsilon}\circ D\Phi_{\zeta}  (\sigma_{\Phi_{-\zeta}(t(\Phi_{\epsilon}(p))}(X(t(\Phi_{\epsilon}(p))))).
\end{eqnarray*}
Clearly, 
\begin{equation*}
\left.\frac{d}{d\epsilon}\right|_{\epsilon=0}W(\epsilon, \epsilon)= \left.\frac{d}{d\epsilon}\right|_{\epsilon=0}W(\epsilon, 0)+\left.\frac{d}{d\epsilon}\right|_{\epsilon=0}W(0,\epsilon).
\end{equation*} 

Thus, we obtain:
\begin{eqnarray*}
Z&=&\left. \frac{d}{d\epsilon}\right|_{\epsilon=0}Dt \circ D\Phi_{-\epsilon}( \sigma_{t(\Phi_{\epsilon}(p))}(X(t(\Phi_{\epsilon}(p)))))
+\left. \frac{d}{d\epsilon}\right|_{\epsilon=0}Dt \circ D\Phi_{\epsilon} ( \sigma_{\Phi_{-\epsilon}(p)}(X(p)))\\
&=&\underbrace{\left. \frac{d}{d\epsilon}\right|_{\epsilon=0}Dt \circ D\Phi_{-\epsilon}(X(t(\Phi_{\epsilon}(p))))}_U
+\underbrace{\left. \frac{d}{d\epsilon}\right|_{\epsilon=0}Dt \circ D\Phi_{\epsilon} ( \sigma_{\Phi_{-\epsilon}(p)}(X(p)))}_V\\
&=& U+V.
\end{eqnarray*}

Here we have used the fact that the connection on $G$ coincides with the canonical one at the identities.
We will now prove that 
\begin{equation*}
\nabla^{bas}_{\alpha}(\beta)=\bar{\Psi}(F_1)_{\alpha}(\beta).
\end{equation*}
Let us set $Q=[\hat{\alpha},\hat{\beta}](p)$, $R=-[\hat{\alpha},\widehat{\rho(\beta)}](p)$ and $S=\bar{\Psi}(F_1)_{\alpha}(\beta)(p)$.
We need to prove that $Q+R=S$.
To this end, we expand $Q$ and $R$:
\begin{eqnarray*}
 Q&=&\left. \frac{d}{d\epsilon}\right|_{\epsilon=0}D\Phi_{-\epsilon} (\beta (\Phi_{\epsilon}(p))),\\
 R&=&-\left. \frac{d}{d\epsilon}\right|_{\epsilon=0}D\Phi_{-\epsilon} (\sigma_{\Phi_{\epsilon}(p)}(\rho(\beta)(p)))
\end{eqnarray*}
and compute $S$:
\begin{eqnarray*}
 S=\bar{\Psi}(F_1)_\alpha(\beta)(p)&=&\left. \frac{d}{d\epsilon}\right|_{\epsilon=0}F_1((\Phi_{\epsilon}(p))^{-1})(\beta(t(\Phi_{\epsilon}(p))))\\
&=&\left. \frac{d}{d\epsilon}\right|_{\epsilon=0}F_1(\Phi_{-\epsilon}(t(\Phi_{\epsilon}(p))))(\beta(t(\Phi_{\epsilon}(p)))).
\end{eqnarray*}
This expression can be written as:
\begin{equation*}
\left.\frac{d}{d\epsilon}\right|_{\epsilon=0}T(\epsilon, \epsilon),
\end{equation*} 
where
\begin{eqnarray*}
T(\epsilon,\zeta)=D\Phi_{-\epsilon}\circ D\Phi_{\zeta}(F_1(\Phi_{-\zeta}(t\Phi_{\epsilon}(p)))(\beta(t(\Phi_{\epsilon}(p))))).
\end{eqnarray*}
Clearly, 
\begin{equation*}
\left.\frac{d}{d\epsilon}\right|_{\epsilon=0}T(\epsilon, \epsilon)= \left.\frac{d}{d\epsilon}\right|_{\epsilon=0}T(\epsilon, 0)+\left.\frac{d}{d\epsilon}\right|_{\epsilon=0}T(0,\epsilon).
\end{equation*} 

Thus,  we obtain:
\begin{eqnarray*}
S&=&\left.\frac{d}{d\epsilon}\right|_{\epsilon=0}D\Phi_{-\epsilon}\circ F_1(t\Phi_{\epsilon}(p))(\beta(t\Phi_{\epsilon}(p)))
+\left.\frac{d}{d\epsilon}\right|_{\epsilon=0}D\Phi_{\epsilon}\circ F_1(\Phi_{-\epsilon}(p))(\beta(p))\\
&=&\underbrace{ \left.\frac{d}{d\epsilon}\right|_{\epsilon=0}D\Phi_{-\epsilon}(\beta(t \Phi_{\epsilon}(p)))}_Q
+\underbrace{\left.\frac{d}{d\epsilon}\right|_{\epsilon=0}D\Phi_{\epsilon}\circ F_1(\Phi_{-\epsilon}(p))(\beta(p))}_{R'}\\
&=&Q+R'.
\end{eqnarray*}
We only need to show that $R=R'$. For this we compute:
\begin{eqnarray*}
 R'&=&\left.\frac{d}{d\epsilon}\right|_{\epsilon=0}D\Phi_{\epsilon}\circ F_1(\Phi_{-\epsilon}(p))(\beta(p))\\
&=& -\left.\frac{d}{d\epsilon}\right|_{\epsilon=0}D\Phi_{\epsilon}\circ \omega_{\Phi_{-\epsilon}(p)} \circ l_{\Phi_{-\epsilon}(p)}(\beta(p))\\
&=& \left.\frac{d}{d\epsilon}\right|_{\epsilon=0}D\Phi_{-\epsilon}\circ \omega_{\Phi_{\epsilon}(p)} \circ l_{\Phi_{\epsilon}(p)}(\beta(p))\\
&=& \left.\frac{d}{d\epsilon}\right|_{\epsilon=0}D\Phi_{-\epsilon}\circ l_{\Phi_{\epsilon}(p)}(\beta(p))
- \left.\frac{d}{d\epsilon}\right|_{\epsilon=0}D\Phi_{-\epsilon}\circ \sigma_{\Phi_{\epsilon}(p)}\circ Ds \circ l_{\Phi_{\epsilon}(p)}(\beta(p))\\
&=&[\hat{\alpha},\check{\beta}](p)- \left.\frac{d}{d\epsilon}\right|_{\epsilon=0}D\Phi_{-\epsilon}\circ \sigma_{\Phi_{\epsilon}(p)}\circ Dt(\beta(p))\\
&=&0+R=R.
\end{eqnarray*}

We are left with proving that $\bar{\Psi}(F_2)=K_{\nabla}$. Instead of computing this directly, we will show
that it is a formal consequence of the previous equations and the naturality of the differentiation process
with respect to pullback, as explained in Remark \ref{lemma:pullback_naturality}. First, we observe that the statement is true for a groupoid with injective anchor
map, as a consequence of the structure equations for a representation up to homotopy.
Using the notation of Subsection \ref{subsection:isomorphism}, consider the Lie groupoid $G^{(0)}$ associated to the action of $G$ on itself, and
the morphism of Lie groupoids $\flat_0: G^{(0)} \rightarrow G$. 
The connection $\sigma$ on $G$ induces naturally a connection $\tilde{\sigma}$ on $G^{(0)}$. Since the anchor map of $G^{(0)}$ is injective, we know
that $\Psi(\Ad_{\tilde{\sigma}}(G^{(0)}))=\ad_{\tilde{\nabla}}(A^{0})$.
Moreover, the natural morphism of the underlying
complexes, given by the derivative of $\flat_0$
\begin{equation*}
D\flat_0:\Ad_{\tilde{\sigma}}(G^{(0)}) \rightarrow {\flat_0}^*(\Ad_{\sigma}(G)),
\end{equation*}
 is a morphism of representations up to homotopy. By differentiating this morphism we obtain:
 \begin{equation*}
D\flat_0: \ad_{\tilde{\nabla}}(A^{(0)})\rightarrow \flat_0^*(\Psi(\Ad_{\sigma}(G))).
 \end{equation*} 
 On the other hand, the same morphism of chain complexes gives also a morphism:
  \begin{equation*}
D\flat_0: \ad_{\tilde{\nabla}}(A^{(0)})\rightarrow \flat_0^*(\ad_{\nabla}(A)).
 \end{equation*} 
Because the map of complexes is surjective, we conclude that:
\begin{equation*}
\flat_0^*(\ad_{\nabla}(A))=\flat_0^*(\Psi(\Ad_{\sigma}(G))).
\end{equation*}
On the other hand, the map $\flat_0$ is a submersion, therefore the pull-back operation is injective and hence $\ad_{\nabla}(A)=\Psi(\Ad_{\sigma}(G))$.

\end{proof}

In \cite{CrainicMoerdijk}, Crainic and Moerdijk introduced a deformation cohomology associated to a Lie algebroid $A$, 
denoted $H_{\mathrm{def}}(A)$. They proved that in degree two this cohomology controls the infinitesimal deformations of the Lie algebroid structure
and stated a rigidity conjecture (Conjecture $1$), which generalizes some rigidity properties of compact Lie groups. 
This conjecture follows from our previous results:

\begin{theorem}\label{conjecture}[Conjecture $1$ of \cite{CrainicMoerdijk}]
If $A$ is a Lie algebroid which admits a proper integrating Lie groupoid $G$ whose fibers are $2$-connected, 
then $H_{\mathrm{def}}^2(A)=0$.
\end{theorem}
\begin{proof}
It is Theorem $3.11$ in \cite{AC1} that for any Lie algebroid $A$  and any connection $\nabla$ on $A$, 
there is a natural isomorphism
\begin{equation*}
H_{\mathrm{def}}(A)\cong H(A,\ad_{\nabla}).
\end{equation*}
Thus, we only need to prove that the cohomology with respect to the adjoint representation vanishes in degree $2$.
Let us chose a connection $\sigma$ on the Lie groupoid $G$ and denote by $\nabla$ the induced connection on $A$.
Then, by Theorem \ref{theorem:differentiation_reps_up_to_homotopy} and Proposition \ref{Ad=ad}, we obtain a map:
\begin{equation*}
\Psi:\hat{C}(G,\Ad_{\sigma})\rightarrow \Omega(A,\Psi(\Ad_\sigma))\cong \Omega(A,\ad_{\nabla}).
\end{equation*}
Since the fibers of $G$ are $2$-connected, Theorem \ref{isomorphism} implies that this map induces an isomorphism in cohomology in degree $2$.
Finally, since $G$ is proper, Theorem $3.35$ of \cite{AC2} implies that $H^k(G,\Ad_\sigma)=0$ for $k>1$. This concludes the proof.

\end{proof}

\end{document}